\documentclass{amsart}
\usepackage[english]{babel}
\usepackage{csquotes}

\usepackage{graphicx}
\usepackage{amsmath}
\usepackage{amssymb}
\usepackage{mathtools}
\usepackage[shortlabels]{enumitem}
\usepackage{amsmath,amsthm,amssymb,amsfonts,mathrsfs,faktor,mathtools,array,bbm}
\usepackage{tikz-cd}
\usetikzlibrary{babel}

\usepackage[hidelinks]{hyperref}
\usepackage[nameinlink,capitalize]{cleveref}
\usepackage{extpfeil}

\usepackage[backend=biber,style=alphabetic]{biblatex}
\newtheorem{theorem}{Theorem}[section] 
\newtheorem{lemma}[theorem]{Lemma}
\newtheorem{proposition}[theorem]{Proposition}
\newtheorem{corollary}[theorem]{Corollary}

\newtheorem{definition}[theorem]{Definition}

\newtheorem{remark}[theorem]{Remark}
\newtheorem{example}[theorem]{Example}

\newtheorem{mainthm}{Main Theorem}

\title{Stability Conditions on Abelian Comma Categories}
\author{Ellen de Oliveira and Guido Neulaender}

\begin{document}
	
	\begin{abstract}
		A comma category, exemplified in algebraic geometry by coherent systems, combines two categories over a third through morphisms between their objects. We establish sufficient conditions for it to be abelian, compute its Grothendieck group, and give necessary and sufficient criteria for it to be noetherian and artinian. Finally, we define a stability condition on abelian comma categories under hypotheses on the initial categories and, conversely, induce stability conditions on the initial abelian categories from those on the comma categories.
		
	\end{abstract}
	
	\maketitle
	
	\tableofcontents
	
	\section{Introduction}
	Comma categories were first introduced by F. W. Lawvere in his 1963 Ph.D. thesis, \textit{ Functorial Semantics of Algebraic Theories} \cite{lawvere1963functorial}. The terminology comes from his notation \((F,G)\), where the comma was used to separate the functors. Lawvere’s thesis remained unpublished until its 2004 reprint, which included the author's commentary \cite{lawvere2004reprints}. In his own words, he introduced comma categories as a natural way ``to show that the notion of adjointness is of an elementary character'' in the theory of the category of categories, making certain constructions more explicit.
	
	
	A comma category, denoted by \((F/G)\), arises from three initial categories \(\mathcal{A}\), \(\mathcal{B}\), and \(\mathcal{C}\), together with two functors \(F: \mathcal{A} \rightarrow \mathcal{C}\) and \(G: \mathcal{B} \rightarrow \mathcal{C}\), relating objects \(A \in \mathcal{A}\) and \(B \in \mathcal{B}\) via a morphism \(\alpha: F(A) \to G(B)\) in \(\mathcal{C}\), forming a triple \((A, B, \alpha)\). Morphisms in \((F/G)\) consist of pairs \((f,g): (A,B,\alpha) \rightarrow (A',B',\alpha')\), with \(f: A \rightarrow A' \in \operatorname{Mor}(\mathcal{A})\) and \(g: B \rightarrow B' \in \operatorname{Mor}(\mathcal{B})\), such that the following diagram commutes:
	\begin{equation*}
		\begin{tikzcd}
			F(A) \arrow[r, "F(f)"] \arrow[d, "\alpha"'] & F(A') \arrow[d, "\alpha'"] \\
			G(B) \arrow[r, "G(g)"']                & G(B')        \end{tikzcd}
	\end{equation*}
	
	
	Although comma categories originated in mathematical logic and universal algebra, our interest in working with them arises naturally from algebraic geometry, where certain decorated sheaves — such as coherent systems (see \cite{he1996espaces}) — appear as particular examples of this type of category. The importance of this categorical approach lies in the fact that, instead of proving the same kind of results separately for each of these specific categories, we can develop a general theory that applies to all of them. 
	
	One of the main results we obtain is a broader notion of stability on comma categories. To achieve this, we first need to ensure that a comma category is abelian and to understand the behavior of its Grothendieck group. In this context, a natural question arises: ``Under what conditions is a comma category abelian?'' Motivated by a discussion in \cite{unapologetic2007comma}, we establish conditions on the categories \(\mathcal{A}\), \(\mathcal{B}\), and \(\mathcal{C}\), as well as on the functors \(F\) and \(G\), that guarantee a comma category is abelian. Having established the abelian structure of a comma category, we then explore its algebraic aspects, focusing on the computation of its Grothendieck group \(K((F/G))\). To characterize \(K((F/G))\), we observe that morphisms in \(\mathcal{C}\) do not affect computation. More precisely, we prove the following result:
	
	\begin{mainthm}[\cref{teo_comma_abeliana} and \cref{teo_GG_comma}]
		\label{mainthm_comma_abel_GG}
		Let $\mathcal{A}$, $\mathcal{B}$, and $\mathcal{C}$ be abelian categories, $F: \mathcal{A} \rightarrow \mathcal{C}$ a right exact functor, and $G: \mathcal{B} \rightarrow \mathcal{C}$ a left exact functor. Then, $(F/G)$ is an abelian category with Grothendieck group \( K((F/G)) \) isomorphic to \( K(\mathcal{A}) \oplus K(\mathcal{B})\).
	\end{mainthm}
	The converse is not true. In fact, we exhibit a counterexample (see \cref{obs_contraexemplo}) in which a comma category \((F/G)\) is abelian, but the functor \(F\) is not right exact.
	
	A stability condition, following \cite{macri2017lectures}, is a pair \((\mathcal{A}, Z)\), where \(Z\) is a stability function and every non-zero object of \(\mathcal{A}\) admits a Harder-Narasimhan (HN) filtration with respect to \(Z\). According to \cite[Proposition 4.10]{macri2017lectures}, if the image of the imaginary part of \(Z\) is discrete for a noetherian category, then every non-zero object admits an HN filtration with respect to \(Z\). First, we obtain necessary and sufficient conditions for \((F/G)\) to be noetherian. Specifically, \((F/G)\) is noetherian if and only if both \(\mathcal{A}\) and \(\mathcal{B}\) are noetherian. Therefore, given stability functions defined on the noetherian categories \(\mathcal{A}\) and \(\mathcal{B}\), we can define a stability function on \((F/G)\) such that the image of its imaginary part is discrete; consequently, this yields a stability condition on \((F/G)\). Conversely, any stability function defined on a noetherian category \((F/G)\), for which the imaginary part has a discrete image, can be naturally restricted to stability functions on \(\mathcal{A}\) and \(\mathcal{B}\). Since the imaginary parts of these restricted functions also have discrete images, it yields stability conditions on \(\mathcal{A}\) and \(\mathcal{B}\). In summary, we obtain the following result:
	
	\begin{mainthm}[\cref{teo_HN_comma} and \cref{teo_HN_sist_cat_inic_HN}]
		Assume the conditions of \cref{mainthm_comma_abel_GG} hold.
		\begin{enumerate}
			\item \((F/G)\) is noetherian if and only if \(\mathcal{A}\) and \(\mathcal{B}\) are noetherian categories.
			
			\item Let \(\mathcal{A}\) and \(\mathcal{B}\) be noetherian categories with \(Z_{\mathcal{A}}\) a stability function on \(\mathcal{A}\) and \(Z_{\mathcal{B}}\) a stability function on \(\mathcal{B}\) such that the image of \(\operatorname{Im}Z_{\mathcal{A}} + \operatorname{Im}Z_{\mathcal{B}}\) is discrete in \(\mathbb{R}\). Then every object in \((F/G)\) admits an HN filtration with respect to the stability function \(Z = Z_{\mathcal{A}} + Z_{\mathcal{B}}\).
			
			\item Let \((F/G)\) be a noetherian category, and let \(Z\) be a stability function on \((F/G)\) such that the image of the imaginary part of \(Z\) is discrete in \(\mathbb{R}\). Then each object in \(\mathcal{A}\) admits an HN filtration with respect to the stability function \(Z_{\mathcal{A}}= Z([(A,0,0)])\) and each object in \(\mathcal{B}\) admits an HN filtration with respect to the stability function \(Z_{\mathcal{B}}= Z([(0,B,0)])\).
		\end{enumerate}
	\end{mainthm}
	
	We also investigate the existence of Jordan-Hölder filtrations on comma categories. According to \cite[Lemma 12.9.6]{stacksproj}, an abelian category has a Jordan-Hölder filtration for every object if and only if it is noetherian and artinian. We obtain a result analogous to the previous one: \((F/G)\) is artinian if and only if both \(\mathcal{A}\) and \(\mathcal{B}\) are artinian. Combining this with our earlier result on the noetherian property, we conclude that the existence of a Jordan-Hölder filtration in \((F/G)\) is equivalent to its existence in \(\mathcal{A}\) and \(\mathcal{B}\). More precisely:
	
	\begin{mainthm}[\cref{teo_JH_A_B_impl_JH_comma} and \cref{teo_JH_comma_impl_JH_A_B}]
		Assume the conditions of \cref{mainthm_comma_abel_GG} hold.
		\begin{enumerate}
			\item  \((F/G)\) is artinian if and only if \(\mathcal{A}\) and \(\mathcal{B}\) are artinian categories.
			
			\item Each object of \(\mathcal{A}\) and \(\mathcal{B}\) admits a Jordan-Hölder filtration if and only if each object of \((F/G)\) admits a Jordan-Hölder filtration.
		\end{enumerate}
	\end{mainthm}

	The paper is organized as follows. In \cref{sect_prelim}, we begin with definitions and propositions about stability in abelian categories, which will be applied in the context of comma categories. Here we follow \cite{lu2013algebraic} for the Grothendieck group, \cite{macri2017lectures} for stability in abelian categories, and \cite{stacksproj} for Jordan-Hölder filtrations. 
	
	In \cref{sect_comma_categories}, we introduce the definition of comma categories (see \cref{def_comma}) and present our motivation for studying them: coherent systems (see \cref{exe_sist_coer}), based on the definition in \cite{he1996espaces}, as a particular example.
	
	In \cref{sect_cond_comma_abel}, we prove propositions ensuring the existence of coproducts, kernels, and cokernels on comma categories, and we present a theorem (see \cref{teo_comma_abeliana}) giving sufficient conditions for a comma category to be abelian. We conclude the section with a counterexample (see \cref{obs_contraexemplo}) showing that the converse does not hold.
	
	In \cref{sect_GG}, we explore monomorphisms, epimorphisms, and isomorphisms on comma categories, which allow us to define short exact sequences in this context. We then observe that morphisms in \(\mathcal{C}\) do not affect the computation of the Grothendieck group of a comma category, and we conclude the section with its explicit calculation (see \cref{teo_GG_comma}), which is given by the direct sum of the Grothendieck groups of \(\mathcal{A}\) and \(\mathcal{B}\). This result is derived by applying the universal property of the Grothendieck group (see \cref{teo_hom_idz_funct_adit}).
	
	In \cref{sect_estab_comma}, we discuss stability in the context of comma categories. We prove necessary and sufficient conditions for a comma category to be noetherian (see \cref{cor_cond_nec_suf_comma_noether}) and, using \cref{prop_exist_HN_func_estab}, we show how to induce HN filtrations in \((F/G)\) under certain conditions on the initial categories (see \cref{teo_HN_comma}). Conversely, under appropriate conditions on comma categories, HN filtrations can be induced on the initial categories (see \cref{teo_HN_sist_cat_inic_HN}).
	
	In \cref{sect_JH_comma}, we study Jordan-Hölder filtrations on comma categories. We give necessary and sufficient conditions for a comma category to be artinian (\cref{cor_cond_nec_suf_comma_artin}) and, using \cref{lem_exts_JH}, we show that a comma category has a Jordan-Hölder filtration if and only if \(\mathcal{A}\) and \(\mathcal{B}\) have Jordan-Hölder filtrations (see \cref{teo_JH_comma_impl_JH_A_B}).
	
	Finally, in \cref{sect_Var_Def_Comma}, inspired by the notion of coherent cosystems in \cite{le1993systemes} (see \cref{def_cos_coe}), we introduce slight modifications to the definition of comma categories to obtain another construction, the co-comma categories, exemplified in algebraic geometry by \textit{framed module} (see \cref{def_framed_module}). Further modifications yield additional variations of comma categories. We conclude the paper by presenting theorems that determine when co-comma categories and these other variations are abelian.

	\section*{Acknowledgments}
	The first author is supported by a Master's scholarship from CNPq grant number 132007/2023-1. The second author is supported by FAPESP, research grant 2023/15556-6. The authors would like to thank their advisor, Prof. Marcos Jardim, for his support throughout the development of this paper and for his invaluable suggestions for its improvement. Finally, the authors would like to thank Exequiel Rivas for pointing out that our categorical definition of coherent systems coincides with the definition of a comma category.

	\section{Preliminaries}
	\label{sect_prelim}
	The basic properties of abelian categories used throughout this text are taken from \cite{freyd1964abelian, grandis2021category, mac2013categories, riehl2017category, rotman2010advanced, rotman2009introduction, weibel2013k, yekutieli2019derived}. We begin by listing some notions from abelian category theory that will be used throughout this work.
	
	\subsection{Grothendieck Group}
	A fundamental concept in the study of stability in abelian categories is the Grothendieck group. Here, we use the references \cite{uitinduced} and \cite{lu2013algebraic}, with particular emphasis on the universal property of the Grothendieck group.
	
	\begin{definition}
		Let $\mathcal{A}$ be an abelian category. The \textbf{Grothendieck group} of $\mathcal{A}$, denoted by \(K(\mathcal{A})\), is the abelian group generated by the isomorphism classes $[A]$ of objects $A \in \mathcal{A}$ satisfying the relation 
		\[
		[A] = [A'] + [A'']
		\]
		for every short exact sequence $0 \rightarrow A' \rightarrow A \rightarrow A'' \rightarrow 0$ in $\mathcal{A}$. 
	\end{definition}
	
	\begin{definition}
		\label{defn_adit_seq_ex_curt}
		Let $\mathcal{A}$ be an abelian category, and let $G$ be an abelian group. A function $f: \operatorname{Obj}(\mathcal{A}) \rightarrow G$ is said to be \textbf{additive} if
		$$f(A) = f(A') + f(A'')$$
		for every short exact sequence $0 \rightarrow A' \rightarrow A \rightarrow A'' \rightarrow 0$ in $\mathcal{A}$.
	\end{definition}
	
	\begin{theorem}[Universal Property, {\cite[Proposition 2.4.14]{lu2013algebraic}}]
		Let \(\mathcal{A}\) be an abelian category, \(G\) an abelian group, and let \(f: \operatorname{Obj}(\mathcal{A}) \rightarrow G\) be an additive function. Then there exists a unique group homomorphism \(\overline{f}: K(\mathcal{A}) \rightarrow G\) such that 
		\[
		f = \overline{f} \circ \eta,
		\]
		where \(\eta: \operatorname{Obj}(\mathcal{A}) \rightarrow K(\mathcal{A})\) is the canonical map defined by \(\eta(A) = [A]\).
		\label{teo_hom_idz_funct_adit}
	\end{theorem}

	\subsection{Stability in Abelian Categories}  
	In what follows, we present the central notions of stability in abelian categories, including key definitions and a fundamental result that will be essential for the development of stability in the context of the comma categories. Our presentation follows \cite{miranda2023moduli}, \cite{macri2017lectures}, and \cite{rudakov1997stability}.
	
	\begin{definition}
		Let \(\mathcal{A}\) be an abelian category, and let \(Z: K(\mathcal{A}) \rightarrow \mathbb{C}\) be an additive homomorphism. We say that \(Z\) is a \textbf{stability function} if, for all non-zero \(A \in \mathcal{A}\), we have:
		\begin{enumerate}
			\item \(\operatorname{Im}Z([A]) \geq 0\);
			
			\item If \(\operatorname{Im}Z([A]) = 0\) then \(\operatorname{Re}Z([A]) < 0\).
		\end{enumerate}
	\end{definition}

	\begin{definition}
		Given a stability function \(Z: K(\mathcal{A}) \rightarrow \mathbb{C}\), we define the \textbf{slope} of an object \(A \in \mathcal{A}\) as  
		\[\mu([A]) := - \dfrac{\operatorname{Re}Z([A])}{\operatorname{Im}Z([A])},\]  
		when \(\operatorname{Im}Z([A]) \neq 0\), and \(\mu([A]) = + \infty\) otherwise.
	\end{definition}
	
	
	\begin{definition}
		Let \(Z: K(\mathcal{A}) \rightarrow \mathbb{C}\) be a stability function. A non-zero object \(A \in \mathcal{A}\) is called \textbf{semistable} (resp. \textbf{stable}) if, for all proper non trivial subobjects \(B \subseteq A\), we have \(\mu([B]) \leq \mu([A])\) (resp. \(\mu([B]) < \mu([A])\)).
	\end{definition}
	
	
	\begin{definition}
		Let \(Z: K(\mathcal{A}) \rightarrow \mathbb{C}\) be an additive homomorphism. We call the pair \((\mathcal{A}, Z)\) a \textbf{stability condition} if
		\begin{enumerate}
			\item \(Z\) is a stability function, and
			
			\item Any non-zero \(A \in \mathcal{A}\) has a filtration, called the \textbf{Harder-Narasimhan filtration},
			\[
			0 = A_{0} \subseteq A_{1} \subseteq \cdots \subseteq A_{n-1} \subseteq A_{n} = A
			\]
			of objects \(A_{i} \in \mathcal{A}\) such that
			\( B_{i} = A_{i}/A_{i-1} \) is semistable for all for \( i = 1, \dots, n \) and \( \mu([B_1]) > \mu([B_2]) > \cdots > \mu([B_n]) \), where \(\mu\) is the slope of \(Z\).
		\end{enumerate}
	\end{definition}
	
	\begin{remark}
		For a stability condition \((\mathcal{A}, Z)\), the Harder-Narasimhan filtration (\emph{HN} filtration) of any object \(A \in \mathcal{A}\) is unique up to isomorphism.
	\end{remark}
	
	\begin{definition}
		Let \(\mathcal{A}\) be an abelian category.
		\begin{enumerate}
			\item  We say an object \( A \) of \( \mathcal{A} \) is \textbf{noetherian} if and only if it satisfies the ascending chain condition for subobjects, i.e., for every ascending chain of subobjects of \(A\)
			\[
			A_{0} \subseteq A_{1} \subseteq \cdots \subseteq A_{i} \subseteq \cdots 
			\]
			there exists \(m \in \mathbb{N}\) such that the monomorphisms \(A_{i} \subseteq A_{i+1}\) are isomorphisms for all \(i \geq m\).
			
			\item We say \( \mathcal{A} \) is \textbf{noetherian} if every object of \( \mathcal{A} \) is noetherian.
		\end{enumerate}
	\end{definition}
	
	\begin{proposition}[{\cite[Proposition 4.10]{macri2017lectures}}]
		Let \( Z : K(\mathcal{A}) \to \mathbb{C} \) be a stability function. Assume that
		\begin{itemize}
			\item \( \mathcal{A} \) is noetherian, and
			\item the image of the imaginary part of \( Z \) is discrete in \( \mathbb{R} \).
		\end{itemize}
		Then every non-zero object in \(\mathcal{A}\) admits a Harder-Narasimhan filtration with respect to \(Z\).
		\label{prop_exist_HN_func_estab}
	\end{proposition}
	
	\subsection{Jordan-Hölder Filtrations}
	Here we use the reference \cite{stacksproj}. 
	\begin{definition}
		Let \(\mathcal{A}\) be an abelian category. An object \(A\) of \(\mathcal{A}\) is said to be \textbf{simple} if it is non-zero and the only subobjects of \(A\) are \(0\) and \(A\).
	\end{definition}
	
	\begin{definition}
		Let \( \mathcal{A} \) be an abelian category and \(A \in \mathcal{A} \).  
		If there exists a filtration
		\[
		0 \subseteq A_{0} \subseteq A_{1} \subseteq \cdots \subseteq A_{n-1} \subseteq A_{n} = A
		\]
		by subobjects such that \( A_i / A_{i-1} \) is simple for \( i = 1, \dots, n \),
		then we say that the filtration \( \{A_i\}_{i=0}^{n} \) is a \textbf{Jordan-Hölder filtration} of length \( n \).
	\end{definition}
	
	\begin{definition}
		Let \(\mathcal{A}\) be an abelian category.
		\begin{enumerate}
			\item  We say an object \( A \) of \( \mathcal{A} \) is \textbf{artinian} if and only if it satisfies the descending chain condition for subobjects, i.e.,  for every descending chain of subobjects of \(A\)
			\[
			\cdots \subseteq A_{i} \subseteq A_{i-1} \subseteq \cdots \subseteq A_{1} \subseteq A_{0}
			\]
			there exists \(m \in \mathbb{N}\) such that the monomorphisms \(A_{i+1} \subseteq A_{i}\) are isomorphisms for all \(i \geq m\).
			
			\item We say \( \mathcal{A} \) is \textbf{artinian} if every object of \( \mathcal{A} \) is artinian.
		\end{enumerate}
	\end{definition}

	\begin{lemma}[{\cite[Lemma 12.9.6]{stacksproj}}]
		Let \( \mathcal{A} \) be an abelian category. Let \( A \) be an object of \( \mathcal{A} \). The following are equivalent:
		\begin{enumerate}
			\item \( A \) is noetherian and artinian, and
			\item There exists a filtration \( 0 \subseteq A_0 \subseteq A_1 \subseteq \dots \subseteq A_n = A \) by subobjects such that \( A_i / A_{i-1} \) is simple for \( i = 1, \dots, n \).
		\end{enumerate}
		\label{lem_exts_JH}
	\end{lemma}
	
			%
	%

	\section{Comma Categories}
	\label{sect_comma_categories}
	Now, we introduce the main motivation for this work, the \textit{comma categories}.
	\begin{definition}
		\label{def_comma}
		Let $\mathcal{A}$, $\mathcal{B}$, and $\mathcal{C}$ be categories, and let $F: \mathcal{A} \rightarrow \mathcal{C}$ and $G: \mathcal{B} \rightarrow \mathcal{C}$ be functors. We define $(F/G)$ such that
		\begin{itemize}
			\item Objects of $(F/G)$: are triples $(A,B,\alpha)$, where $A \in \operatorname{Obj}(\mathcal{A})$, $B \in \operatorname{Obj}(\mathcal{B})$ and $\alpha: F(A) \rightarrow G(B) \in \operatorname{Mor}(\mathcal{C})$.
			
			\item Morphisms of $(F/G)$: a morphism between $(A,B,\alpha)$ and $(A',B',\alpha')$, objects of \((F/G)\), is a pair $(f,g)$ with $f: A \rightarrow A' \in \operatorname{Mor}(\mathcal{A})$, $g: B \rightarrow B' \in \operatorname{Mor}(\mathcal{B})$ such that the following diagram commutes:
			\begin{equation*}
				\begin{tikzcd}
					F(A) \arrow[r, "F(f)"] \arrow[d, "\alpha"'] & F(A') \arrow[d, "\alpha'"] \\
					G(B) \arrow[r, "G(g)"']                & G(B')        \end{tikzcd}
			\end{equation*}
			that is,
			\begin{equation*}
				\alpha' \circ F(f) = G(g) \circ \alpha.
			\end{equation*}
		\end{itemize}
	\end{definition}
	
	First, we show that $(F/G)$ is indeed a category. The proof is divided into two parts:
	\begin{enumerate}
		\item Let $(A,B,\alpha) \in \operatorname{Obj}(F/G)$ be any object, so $\alpha: F(A) \rightarrow G(B)$. We define the identity morphism by $1_{(A,B,\alpha)} = (1_A, 1_B)$. Note that
		\begin{equation}
			\alpha \circ F(1_{A}) = \alpha \circ 1_{F(A)} = \alpha = 1_{G(B)} \circ \alpha = G(1_{B}) \circ \alpha .
			\label{alfa_iso_comma}
		\end{equation}
		By \cref{alfa_iso_comma}, the following diagram commutes:
		\begin{equation*}
			\begin{tikzcd}
				F(A) \arrow[r, "F(1_{A})"] \arrow[d, "\alpha"'] & F(A) \arrow[d, "\alpha"] \\
				G(B) \arrow[r, "G(1_{B})"']                     & G(B)                    
			\end{tikzcd}
		\end{equation*}
		Thus, $1_{(A,B,\alpha)}$ is a morphism in $(F/G)$.
		
		\item Let $(f,g): (A,B,\alpha) \rightarrow (A',B',\alpha')$ and $(f',g'): (A',B',\alpha') \rightarrow (A'',B'',\alpha'')$ be morphisms in $(F/G)$. In this case, the following diagrams commute:
		\begin{equation*}
			\begin{tikzcd}
				F(A) \arrow[r, "F(f)"] \arrow[d, "\alpha"'] & F(A') \arrow[d, "\alpha'"] &  & F(A') \arrow[r, "F(f')"] \arrow[d, "\alpha'"'] & F(A'') \arrow[d, "\alpha''"] \\ G(B) \arrow[r, "G(g)"']   & G(B')     &  & G(B') \arrow[r, "G(g')"']                      & G(B'')   
			\end{tikzcd}
		\end{equation*}
		that is,
		\begin{equation}
			\alpha' \circ F(f) = G(g) \circ \alpha .
			\label{comma_alpha}
		\end{equation}
		\begin{equation}
			\alpha'' \circ F(f') = G(g') \circ \alpha' .
			\label{comma_alpha'}
		\end{equation}
		We define the composition in $(F/G)$ by
		$$ (f',g') \circ (f,g) := (f' \circ f, g' \circ g). $$
		To prove it is well defined, we must show that $(f',g') \circ (f,g)$ is a morphism in $(F/G)$, that the operation is associative, and that the composition with the identities is trivial.
		The first follows from the commutativity of the diagram:
		\begin{equation*}
			\begin{tikzcd}
				F(A) \arrow[d, "\alpha"'] \arrow[rr, "F(f' \circ f)"] &  & F(A'') \arrow[d, "\alpha''"] \\
				G(B) \arrow[rr, "G(g' \circ g)"']   &  & G(B'')                  \end{tikzcd}
		\end{equation*}
		In fact, by \cref{comma_alpha'} and \cref{comma_alpha}
		\begin{eqnarray*}
			\alpha'' \circ F(f' \circ f) 
			&=&
			\alpha'' \circ (F(f') \circ F(f)) = (\alpha'' \circ F(f')) \circ F(f) 
			\\
			&=&
			(G(g') \circ \alpha') \circ F(f) =
			G(g')  \circ (\alpha' \circ F(f))  
			\\
			&=&
			G(g') \circ (G(g) \circ \alpha) = G(g' \circ g) \circ \alpha.
		\end{eqnarray*}
		The latter two can be seen as follows.
		\begin{enumerate}
			\item Consider $(f,g): (A,B,\alpha) \rightarrow (A',B',\alpha')$, $(f',g'): (A',B',\alpha') \rightarrow (A'',B'',\alpha'')$ and $(f'',g''): (A'',B'',\alpha'') \rightarrow (A''',B''',\alpha''')$ morphisms in \((F/G)\). Thus,
			\begin{eqnarray*}
				(f'',g'') \circ ((f',g') \circ (f,g)) 
				&=&
				(f'',g'') \circ (f' \circ f, g' \circ g) \\
				&=&
				(f'' \circ (f' \circ f),g'' \circ (g' \circ g)) \\
				&=& 
				((f'' \circ f') \circ f, (g'' \circ g') \circ g) \\
				&=&
				(f'' \circ f', g'' \circ g') \circ (f,g) \\
				&=&
				((f'',g'') \circ (f',g')) \circ (f,g) .
			\end{eqnarray*}
			
			\item Let $(f,g): (A,B,\alpha) \rightarrow (A',B',\alpha')$ be a morphism in $(F/G)$. Then,
			$$        (f,g) \circ 1_{(A,B,\alpha)} = (f,g) \circ (1_{A},1_{B}) = (f \circ 1_{A}, g \circ 1_{B}) = (f,g).$$
			$$      1_{(A',B',\alpha')} \circ (f,g)  = (1_{A'},1_{B'}) \circ (f,g)  = (1_{A'} \circ f, 1_{B'} \circ g) = (f,g).$$
		\end{enumerate}    
	\end{enumerate}
	Therefore, \((F/G)\) forms a category, called a \textbf{comma category}
	
	A concrete example of a comma category, which motivated our interest in studying it, is the category of coherent systems described in \cite{he1996espaces}.
	\begin{example}[Coherent Systems]
		\label{exe_sist_coer}
		Let \(X\) be an algebraic variety. An \textbf{algebraic system} on \(X\) is a triple \(\Lambda = (\Gamma, \sigma, \mathscr{F})\), 
		where \( \mathscr{F} \) is an \( \mathcal{O}_X \)-module, \( \Gamma \) is a \( \mathbb{C} \)-vector space, and \( \sigma : \Gamma  \rightarrow H^0(\mathscr{F}) \) is a \( \mathbb{C} \)-linear application. A \textbf{morphism of algebraic systems} on \(X\) is a pair of morphisms
		\(\Theta = (\alpha, \beta) : (\Gamma, \sigma, \mathscr{F}) \rightarrow (\Gamma', \sigma', \mathscr{F}')
		\), where \( \beta : \mathscr{F} \rightarrow \mathscr{F}' \) is a morphism of \( \mathcal{O}_X \)-modules and \( \alpha : \Gamma \rightarrow \Gamma' \) is a \( \mathbb{C} \)-linear morphism such that \(\sigma' \circ \alpha = H^0(\beta) \circ \sigma\). 
		
		An algebraic system \( \Lambda = (\Gamma, \sigma, \mathscr{F}) \) on \( X \) is called a \textbf{coherent system} if \( \mathscr{F} \) is a coherent \( \mathcal{O}_X \)-module and if \( \Gamma \) is of finite dimension. The morphisms of coherent systems are the morphisms of algebraic systems.
		
		We can rewrite the definition of coherent systems using the language of comma categories. To do so, consider \(\mathcal{A} = \operatorname{Vect}^{f}_{\mathbb{C}}\), \(\mathcal{B}=\operatorname{Coh}(X)\) and \(\mathcal{C} = \operatorname{Vect}^{f}_{\mathbb{\mathbb{C}}}\); \(F = \operatorname{Id}\) and \(G = H^{0}\). A coherent system is a triple \((\Gamma, \mathscr{F}, \sigma)\) with \(\Gamma \in \mathcal{A}\), \(\mathscr{F} \in \mathcal{B}\) and \(\sigma: F(\Gamma) \rightarrow G(\mathscr{F}) \in \operatorname{Mor}(\mathcal{C})\). A morphism between coherent systems is a pair \((\alpha, \beta): (\Gamma, \mathscr{F}, \sigma) \rightarrow (\Gamma', \mathscr{F}', \sigma')\), with \(\alpha: \Gamma \rightarrow \Gamma' \in \operatorname{Mor}(\mathcal{A})\) and \(\beta: \mathscr{F} \rightarrow \mathscr{F}' \in \operatorname{Mor}(\mathcal{B})\), such that \(\sigma' \circ F(\alpha) = G(\beta) \circ \sigma\).
	\end{example}

	After defining a category, a natural question arises: ``\textit{When is this category abelian?}''. Our next step is to establish conditions for a comma category to be an abelian category.

	\section{Conditions for Comma Categories to be Abelian}
	\label{sect_cond_comma_abel}
	In this section, we present the results that answer our first question, namely, the conditions on the categories and functors involved in the definition of a comma category for it to be abelian.
	\begin{proposition}
		Let $\mathcal{A}$, $\mathcal{B}$, and $\mathcal{C}$ be additive categories and let $F: \mathcal{A} \rightarrow \mathcal{C}$ and $G: \mathcal{B} \rightarrow \mathcal{C}$ be additive functors. Then $(F/G)$ is an additive category.
		\label{prop_comma_aditiva}
	\end{proposition}
	\begin{proof}
		We will show that $(F/G)$ is linear. Let $(A, B, \alpha)$ and $(A', B', \alpha')$ be arbitrary objects in $(F/G)$. From the fact that $\mathcal{A}$ and $\mathcal{B}$ are linear, for morphisms
		$(f,g),(f',g') \in \operatorname{Hom}_{(F/G)}((A,B,\alpha),(A',B',\alpha'))$ define
		$$(f,g) + (f',g') := (f +_{\mathcal{A}} f', g +_{\mathcal{B}} g').$$
		We will show that $(f, g) + (f', g')$ is a morphism in $(F/G)$, that is, that the diagram below commutes.
		\begin{equation*}
			\begin{tikzcd}
				F(A) \arrow[rr, "F(f +_{\mathcal{A}}f')"] \arrow[d, "\alpha"'] &  & F(A') \arrow[d, "\alpha'"] \\
				G(B) \arrow[rr, "G(g+_{\mathcal{B}}g')"']                      &  & G(B')                     
			\end{tikzcd}
		\end{equation*}
		Since $(f, g)$ and $(f', g')$ are morphisms in $(F/G)$, the diagrams below commute.
		\begin{equation*}
			\begin{tikzcd}
				F(A) \arrow[r, "F(f)"] \arrow[d, "\alpha"'] & F(A') \arrow[d, "\alpha'"] & F(A) \arrow[r, "F(f')"] \arrow[d, "\alpha"'] & F(A') \arrow[d, "\alpha'"] \\ G(B) \arrow[r, "G(g)"'] & G(B')  & G(B) \arrow[r, "G(g')"'] & G(B')                     
			\end{tikzcd}
		\end{equation*}
		so,
		\begin{equation}
			\label{alfa_alfa'_mor_comm}
			\alpha' \circ F(f) = G(g) \circ \alpha \ \ \ \mbox{ and } \ \ \     \alpha' \circ F(f') = G(g') \circ \alpha .
		\end{equation}
		Since  \(\mathcal{C}\) is linear and \(F\) and \(G\) are additive functors, from equation \cref{alfa_alfa'_mor_comm} we obtain
		\begin{eqnarray*}
			\alpha' \circ F(f +_{\mathcal{A}}f') 
			&=&
			\alpha' \circ (F(f) +_{\mathcal{C}} F(f')) 
			\\
			&=&
			(\alpha' \circ F(f)) +_{\mathcal{C}} (\alpha' \circ F(f')) 
			\\
			&=&
			(G(g) \circ \alpha ) +_{\mathcal{C}} (G(g') \circ  \alpha) 
			\\ 
			&=&
			(G(g) +_{\mathcal{C}} G(g')) \circ \alpha \\
			&=&
			G(g +_{\mathcal{B}} g') \circ \alpha.
		\end{eqnarray*}
		For a scalar \(n \in \mathbb{Z}\) and \((f,g) \in \operatorname{Hom}_{(F/G)}((A,B,\alpha),(A',B',\alpha'))\), define
		\[
		n(f,g):= (nf,ng).
		\]
		Since  \(\mathcal{C}\) is linear and \(F\) and \(G\) are additive functors, from equation \cref{alfa_alfa'_mor_comm} we obtain
		\[
		\alpha' \circ F(nf) = \alpha' \circ nF(f) = n(\alpha' \circ F(f)) = n(G(g) \circ \alpha) = nG(g) \circ \alpha = G(ng) \circ \alpha,\]
		so \(n(f,g) \in \operatorname{Hom}_{(F/G)}((A,B,\alpha),(A',B',\alpha'))\). Based on the addition and scalar multiplication defined above, and noting that \(\operatorname{Hom}_{\mathcal{A}}(A, A')\) and \(\operatorname{Hom}_{\mathcal{B}}(B, B')\) are \(\mathbb{Z}\)-modules, we conclude that \(\operatorname{Hom}_{(F/G)}((A, B, \alpha), (A', B', \alpha'))\) is also a \(\mathbb{Z}\)-module.
		
		Let \((f, g), (f', g') \in \operatorname{Hom}_{(F/G)}((A,B,\alpha), (A',B',\alpha'))\). Consider also \((\overline{f}, \overline{g}) \in \operatorname{Hom}_{(F/G)}((A',B',\alpha'), (A'',B'',\alpha''))\).
		Thus,
		
		\begin{eqnarray*}
			(\overline{f},\overline{g}) \circ ((f,g) + (f',g')) 
			&=&
			(\overline{f},\overline{g}) \circ (f +_{\mathcal{A}} f', g +_{\mathcal{B}} g') 
			\\
			&=&(\overline{f} \circ (f +_{\mathcal{A}} f'), \overline{g} \circ (g +_{\mathcal{B}} g')) \\
			&=&
			(\overline{f} \circ f +_{\mathcal{A}} \overline{f} \circ  f', \overline{g} \circ g +_{\mathcal{B}} \overline{g} \circ  g') 
			\\
			&=&(\overline{f} \circ f, \overline{g} \circ g) +
			(\overline{f} \circ  f' , \overline{g} \circ  g') 
			\\
			&=& (\overline{f},\overline{g}) \circ (f,g) + (\overline{f},\overline{g}) \circ (f',g')   .
		\end{eqnarray*}
		Similarly
		$$((\overline{f},\overline{g}) + (f,g)) \circ  (f',g') = (\overline{f},\overline{g}) \circ (f',g') + (f,g) \circ (f',g'),$$
		for morphisms in \((F/G)\) for which the composition is defined.
		This way, \((F/G)\) is a linear category.
		
		Now, we will show that $(F/G)$ has a zero object. Let $0_{\mathcal{A}}$, $0_{\mathcal{B}}$, and $0_{\mathcal{C}}$ be the zero objects of $\mathcal{A}$, $\mathcal{B}$, and $\mathcal{C}$, respectively. By property of additive functors we have \(F(0_{\mathcal{A}}) = 0_{\mathcal{C}} = G(0_{\mathcal{B}})\). Furthermore, consider
		$0 \in \operatorname{Hom}_{\mathcal{C}}(0_{\mathcal{C}},0_{\mathcal{C}})$. We will prove that $(0_{\mathcal{A}}, 0_{\mathcal{B}}, 0)$ is the zero object of $(F/G)$. Indeed, let $(A, B, \alpha)$ be an arbitrary object of $(F/G)$. Since $0_{\mathcal{A}}$, $0_{\mathcal{B}}$, and $0_{\mathcal{C}}$ are initial objects, there exist unique morphisms \(f: 0_{\mathcal{A}} \rightarrow A \in \operatorname{Mor}(\mathcal{A}) \), \(g: 0_{\mathcal{B}} \rightarrow B \in \operatorname{Mor}(\mathcal{B})\), and \(h: 0_{\mathcal{C}} \rightarrow G(B) \in \operatorname{Mor}(\mathcal{C}) \). This implies that the diagram below commutes.
		\begin{equation*}
			\begin{tikzcd}
				0_\mathcal{C} = F(0_{\mathcal{A}}) \arrow[r, "F(f)"] \arrow[d, "0"'] \arrow[rd, "h", dashed] & F(A) \arrow[d, "\alpha"] \\ 0_\mathcal{C} = G(0_{\mathcal{B}}) \arrow[r, "G(g)"'] & G(B)  
			\end{tikzcd}
		\end{equation*}
		Thus, $(f, g): (0_{\mathcal{A}}, 0_{\mathcal{B}}, 0) \rightarrow (A, B, \alpha)$ is a morphism in \((F/G)\). Moreover, if \((\Tilde{f},\Tilde{g}): (0_{\mathcal{A}}, 0_{\mathcal{B}}, 0) \rightarrow (A, B, \alpha)\) is a morphism in \((F/G)\) we have \(\Tilde{f}: 0_{\mathcal{A}} \rightarrow A\), thus \(\Tilde{f} = f\) because \( 0_{\mathcal{A}}\) is an initial object in \(\mathcal{A}\) and \(\Tilde{g}: 0_{\mathcal{B}} \rightarrow B\), thus \(\Tilde{g} = g\) because \( 0_{\mathcal{B}}\) is an initial object in \(\mathcal{B}\), i.e., \((\Tilde{f},\Tilde{g}) = (f,g)\). This way, $(0_{\mathcal{A}}, 0_{\mathcal{B}}, 0)$ is an initial object in \((F/G)\). Similarly, it can be shown that $(0_{\mathcal{A}}, 0_{\mathcal{B}}, 0)$ is a final object. Hence, $(0_{\mathcal{A}}, 0_{\mathcal{B}}, 0)$ is both initial and final in \((F/G)\), and thus a zero object.
		
		Finally, we will show the existence of finite coproducts in $(F/G)$. Let $(A, B, \alpha)$ and $(A', B', \alpha')$ be arbitrary objects in $(F/G)$. By assumption, $\mathcal{A}$ and $\mathcal{B}$ admit finite coproducts, so consider $(A \oplus A', i_{1}^{A}, i_{2}^{A})$ the coproduct of $A$ and $A'$, and $(B \oplus B', i_{1}^{B}, i_{2}^{B})$ the coproduct of $B$ and $B'$. Since additive functors preserve coproducts, it follows that \((F(A \oplus A'), F(i_{1}^{A}), F(i_{2}^{A}))\) is the coproduct of \(F(A)\) and \(F(A')\) and \((G(B \oplus B'), G(i_{1}^{B}), G(i_{2}^{B}))\) is the coproduct of \(G(B)\) and \(G(B')\).
		Thus, we have the following diagram
		\begin{equation*}
			\begin{tikzcd}
				F(A) \arrow[r, "F(i_{1}^{A})"] \arrow[rd, "G(i_{1}^{B}) \circ \alpha"'] & F(A \oplus A') \arrow[d, "\beta", dashed] & F(A') \arrow[l, "F(i_{2}^{A})"'] \arrow[ld, "G(i_{2}^{B}) \circ \alpha'"] \\ & G(B \oplus B')  & 
			\end{tikzcd}
		\end{equation*}
		where $\beta$ comes from the universal property of the coproduct of $F(A)$ and $F(A')$. Thus, $\beta$ is unique such that
		\begin{equation}
			G(i_{1}^{B}) \circ \alpha = \beta \circ F(i_{1}^{A}) \ \mbox{ and } \ G(i_{2}^{B}) \circ  \alpha' = \beta \circ F(i_{2}^{A}) .
			\label{eq_beta_cop_comma}
		\end{equation}
		From \cref{eq_beta_cop_comma} we have that the diagram below commutes.
		\begin{equation*}
			\begin{tikzcd}
				F(A) \arrow[d, "\alpha"'] \arrow[r, "F(i_{1}^{A})"] & F(A \oplus A') \arrow[d, "\beta"] &  & F(A') \arrow[d, "\alpha'"'] \arrow[r, "F(i_{2}^{A})"] & F(A \oplus A') \arrow[d, "\beta"] \\
				G(B) \arrow[r, "G(i_{1}^{B})"']                     & G(B \oplus B')                    &  & G(B') \arrow[r, "G(i_{2}^{B})"']                      & G(B \oplus B')                   
			\end{tikzcd}
		\end{equation*}
		Thus \((i_{1}^{A}, i_{1}^{B}): (A,B, \alpha) \rightarrow (A \oplus A', B \oplus B',\beta)\)  
		and  
		\((i_{2}^{A}, i_{2}^{B}): (A',B', \alpha') \rightarrow (A \oplus A', B \oplus B',\beta)\)  
		are morphisms in \((F/G)\).
		
		We will show that $((A \oplus A', B \oplus B', \beta), (i_{1}^{A}, i_{1}^{B}), (i_{2}^{A}, i_{2}^{B}))$ is the coproduct of $(A, B, \alpha)$ and $(A', B', \alpha')$. Indeed, let $(X, Y, H)$ be an object of $(F/G)$ and let $(f_{1}, g_{1}): (A, B, \alpha) \rightarrow (X, Y, H)$ and $(f_{2}, g_{2}): (A', B', \alpha') \rightarrow (X, Y, H)$ be arbitrary morphisms in $(F/G)$. We want to find a unique morphism in \((F/G)\) given by $(\varphi, \psi): (A \oplus A', B \oplus B', \beta) \rightarrow (X, Y, H)$ such that the diagram below commutes:
		\begin{equation*}
			\begin{tikzcd}
				&  & {(X,Y,H)}  &  &  \\
				{(A,B,\alpha)} \arrow[rru, "{(f_{1},g_{1})}"] \arrow[rr, "{(i_{1}^{A},i_{1}^{B})}"'] &  & {(A \oplus A, B \oplus B', \beta)} \arrow[u, "{(\varphi,\psi)}"', dashed] &  & {(A',B',\alpha')} \arrow[llu, "{(f_{2},g_{2})}"'] \arrow[ll, "{(i_{2}^{A},i_{2}^{B})}"]
			\end{tikzcd}
		\end{equation*}
		First, we will show the existence of the morphism $(\varphi, \psi)$. From the fact that $(f_{1}, g_{1})$ and $(f_{2}, g_{2})$ are morphisms, and using the universal property of the coproducts of $A$, $A'$ and $B$, $B'$, there exist unique morphisms $\varphi: A \oplus A' \rightarrow X$ and $\psi: B \oplus B' \rightarrow Y$ such that the diagrams below commute.
		\begin{equation*}
			\begin{tikzcd}
				& X  &  &  & Y  &    \\
				A \arrow[ru, "f_{1}"] \arrow[r, "i_{1}^{A}"'] & A \oplus A'  \arrow[u, "\varphi"', dashed] & A' \arrow[lu, "f_{2}"'] \arrow[l, "i_{2}^{A}"] & B \arrow[ru, "g_{1}"] \arrow[r, "i_{1}^{B}"'] & B \oplus B' \arrow[u, "\psi"', dashed] & B' \arrow[lu, "g_{2}"'] \arrow[l, "i_{2}^{B}"]
			\end{tikzcd}
		\end{equation*}
		that is,
		\begin{equation}
			f_{1} = \varphi \circ i_{1}^{A} \ \mbox{ and } \ f_{2} = \varphi \circ i_{2}^{A} .
			\label{eq_varphi_comma}
		\end{equation}
		\begin{equation}
			g_{1} = \psi \circ i_{1}^{B} \ \mbox{ and } \ g_{2} = \psi \circ i_{2}^{B} .
			\label{eq_psi_comma}
		\end{equation}
		This implies that
		\begin{equation}
			F(f_{1}) = F(\varphi) \circ F(i_{1}^{A}) \ \mbox{ and } \ F(f_{2}) = F(\varphi) \circ F(i_{2}^{A}) .
			\label{eq_Fvarphi_comma}
		\end{equation}
		\begin{equation}
			G(g_{1}) = G(\psi) \circ G(i_{1}^{B}) \ \mbox{ and } \ G(g_{2}) = G(\psi) \circ G(i_{2}^{B}) .
			\label{eq_Gpsi_comma}
		\end{equation}
		By the universal property of the coproduct $(F(A \oplus A'), F(i_{1}^{A}), F(i_{2}^{A}))$ of \(F(A)\) and \(F(A')\), there exist unique morphisms $\rho, \rho': F(A \oplus A') \rightarrow G(Y)$ such that the diagrams below commute.
		\begin{equation*}
			\begin{tikzcd}
				& G(Y) & & & G(Y) &\\
				F(A) \arrow[ru, "H \circ F(f_{1})"] \arrow[r, "F(i_{1}^{A})"'] & F(A \oplus A') \arrow[u, "\rho"', dashed] & F(A') \arrow[lu, "H \circ F(f_{2})"'] \arrow[l, "F(i_{2}^{A})"] & F(A) \arrow[ru, "G(g_{1}) \circ \alpha"] \arrow[r, "F(i_{1}^{A})"'] & F(A \oplus A') \arrow[u, "\rho'", dashed] & F(A') \arrow[lu, "G(g_{2}) \circ \alpha'"'] \arrow[l, "F(i_{2}^{A})"]
			\end{tikzcd}
		\end{equation*}
		i.e.,
		\begin{equation}
			H \circ F(f_{1}) = \rho \circ F(i_{1}^{A}) \ \mbox{ and } \  H \circ F(f_{2}) = \rho \circ F(i_{2}^{A}) .
			\label{phi_HF_comma}
		\end{equation}
		\begin{equation}
			G(g_{1}) \circ \alpha  = \rho' \circ F(i_{1}^{A}) \ \mbox{ and } \   G(g_{2}) \circ \alpha' = \rho' \circ F(i_{2}^{A}) .
			\label{phi'_G_comma}
		\end{equation}
		By \cref{eq_Fvarphi_comma}
		\begin{equation}
			H \circ F(f_{1}) =
			H \circ ( F(\varphi) \circ F(i_{1}^{A})) = (H \circ F(\varphi)) \circ F(i_{1}^{A}) .
			\label{unic_phi1}
		\end{equation}
		\begin{equation}
			H \circ F(f_{2}) =
			H \circ ( F(\varphi) \circ F(i_{2}^{A})) = (H \circ F(\varphi)) \circ F(i_{2}^{A}) .
			\label{unic_phi2}
		\end{equation}
		From the \cref{phi_HF_comma}, \cref{unic_phi1}, \cref{unic_phi2}, and by the uniqueness of $\rho$, we obtain
		\begin{equation}
			\rho = H \circ F(\varphi) .
			\label{phiF}
		\end{equation}
		Similarly, by \cref{eq_Gpsi_comma} and \cref{eq_beta_cop_comma}
		\begin{equation}
			G(g_{1}) \circ \alpha =
			(G(\psi) \circ G(i_{1}^{B})) \circ \alpha 
			=
			G(\psi) \circ (\beta \circ F(i_{1}^{A})) = (G(\psi) \circ \beta) \circ  F(i_{1}^{A}) .
			\label{unic_phi'1}
		\end{equation}
		\begin{equation}
			G(g_{2}) \circ \alpha' =
			(G(\psi) \circ G(i_{2}^{B})) \circ \alpha'  =
			G(\psi) \circ (\beta \circ F(i_{2}^{A})) = (G(\psi) \circ \beta) \circ  F(i_{2}^{A}) .
			\label{unic_phi'2}
		\end{equation}
		From the \cref{phi'_G_comma}, \cref{unic_phi'1}, \cref{unic_phi'2}, and by the uniqueness of $\rho'$, it follows that
		\begin{equation}
			\rho' = G(\psi) \circ \beta.
			\label{phi'G}
		\end{equation}
		From the fact that $(f_{1}, g_{1})$ and $(f_{2}, g_{2})$ are morphisms in $(F/G)$, we have
		\begin{equation}
			H \circ F(f_{1}) = G(g_{1}) \circ \alpha.
			\label{f1f2_morf}
		\end{equation}
		\begin{equation}
			H \circ F(f_{2}) = G(g_{2}) \circ \alpha'.
			\label{g1g2_morf}
		\end{equation}
		From the  \cref{f1f2_morf}, \cref{g1g2_morf}, \cref{unic_phi'1}, \cref{unic_phi'2} and \cref{phi'G}, we obtain that $H \circ F(f_{1}) = \rho' \circ F(i_{1}^{A})$ and $H \circ F(f_{2}) = \rho' \circ F(i_{2}^{A})$. By the uniqueness of $\rho$, we have $\rho = \rho'$. Thus, from the \cref{phiF} and \cref{phi'G}, it follows that $H \circ F(\varphi) = G(\psi) \circ \beta$, that is, the diagram below commutes.
		\begin{equation*}
			\begin{tikzcd}
				F(A \oplus A') \arrow[r, "F(\varphi)"] \arrow[d, "\beta"'] & F(X) \arrow[d, "H"] \\ G(B \oplus B') \arrow[r, "G(\psi)"'] & G(Y)
			\end{tikzcd}
		\end{equation*}
		Thus, $(\varphi, \psi)$ is a morphism in $(F/G)$. Moreover, by \cref{eq_varphi_comma} and \cref{eq_psi_comma} we have
		\[
		(\varphi,\psi) \circ (i_{1}^{A},i_{1}^{B}) = (\varphi \circ i_{1}^{A}, \psi \circ i_{1}^{B}) = (f_{1}, g_{1}) .
		\] 
		\[
		(\varphi,\psi) \circ (i_{2}^{A},i_{2}^{B}) = (\varphi \circ i_{2}^{A}, \psi \circ i_{2}^{B}) = (f_{2}, g_{2}).
		\]
		Finally, suppose $(\Tilde{\varphi}, \Tilde{\psi}): (A \oplus A', B \oplus B', \beta) \rightarrow (X, Y, H)$ is a morphism in $(F/G)$ such that 
		\[(\Tilde{\varphi}, \Tilde{\psi}) \circ (i_{1}^{A},i_{1}^{B}) =(f_{1}, g_{1}) \ \mbox{ and } \ (\Tilde{\varphi}, \Tilde{\psi}) \circ (i_{2}^{A},i_{2}^{B}) = (f_{2}, g_{2}). \]
		Then,
		\[\Tilde{\varphi} \circ i_{1}^{A} = f_{1} , \Tilde{\varphi} \circ i_{2}^{A} = f_{2} \ \mbox{ and } \  \Tilde{\psi} \circ i_{1}^{B} = g_{1} ,  \Tilde{\psi} \circ i_{2}^{B} = g_{2}. \]
		From the uniqueness of \(\varphi\) and \(\psi\), we have \(\Tilde{\varphi} = \varphi\) and \(\Tilde{\psi} = \psi\). This way, \((F/G)\) has finite coproducts. Therefore, \((F/G)\) is an additive category.
	\end{proof}
	
	\begin{proposition}
		Let $\mathcal{A}$, $\mathcal{B}$, and $\mathcal{C}$ be abelian categories, $F: \mathcal{A} \rightarrow \mathcal{C}$ an additive functor, and $G: \mathcal{B} \rightarrow \mathcal{C}$ a left exact functor. Then, $(F/G)$ admits kernels.
		\label{prop_comma_kernel}
	\end{proposition}
	\begin{proof}
		Let $(f, g): (A, B, \alpha) \rightarrow (A', B', \alpha')$ be a morphism in $(F/G)$. Since $f: A \rightarrow A' \in \operatorname{Mor}(\mathcal{A})$ and $g: B \rightarrow B' \in \operatorname{Mor}(\mathcal{B})$, there exist $(\operatorname{Ker} f, \operatorname{ker} f)$ and $(\operatorname{Ker} g, \operatorname{ker} g)$. Since $(f, g)$ is a morphism in $(F/G)$, we have
		\begin{equation}
			\alpha' \circ F(f) = G(g) \circ \alpha .
			\label{fg_morf_comma}
		\end{equation}
		Hence, by \cref{fg_morf_comma}
		\begin{equation*}
			G(g) \circ \alpha \circ  F(\operatorname{ker}f) =
			\alpha' \circ F(f) \circ F(\operatorname{ker}f) = \alpha' \circ F(f \circ \operatorname{ker}f) = \alpha' \circ F(0) = 0 .
		\end{equation*}
		From the left exactness of $G$ we have \(\operatorname{Ker}G(g) = G(\operatorname{Ker}g)\) and \(\operatorname{ker}G(g) = G(\operatorname{ker}g)\). Moreover, using the universal property of the kernel of $G(g)$, there exists a unique morphism $\beta: F(\operatorname{Ker} f) \rightarrow G(\operatorname{Ker} g)$ such that the diagram below commutes.
		\begin{equation*}
			\begin{tikzcd}
				G(\operatorname{Ker}g) \arrow[r, "G(\operatorname{ker}g)"] & G(B) \arrow[r, "G(g)"] & G(B') \\ & F(\operatorname{Ker}f) \arrow[u, "\alpha \circ F(\operatorname{ker}f)"'] \arrow[lu, "\beta", dashed] &  
			\end{tikzcd}
		\end{equation*}
		i.e.,
		\begin{equation}
			\alpha \circ F(\operatorname{ker}f) = G(\operatorname{ker}g) \circ \beta .
			\label{beta_G_exac_esq}
		\end{equation}
		From \cref{beta_G_exac_esq}, we have that the diagram below commutes.
		\begin{equation*}
			\begin{tikzcd}
				F(\operatorname{Ker}f) \arrow[r, "F(\operatorname{ker}f)"] \arrow[d, "\beta"'] & F(A) \arrow[d, "\alpha"] \\ G(\operatorname{Ker}g) \arrow[r, "G(\operatorname{ker}g)"']                    & G(B)    
			\end{tikzcd}
		\end{equation*}
		Then $(\operatorname{ker} f, \operatorname{ker} g)$ is a morphism in $(F/G)$.
		We will show that the kernel of the morphism \((f, g)\) is given by \(((\operatorname{Ker} f, \operatorname{Ker} g, \beta), (\operatorname{ker} f, \operatorname{ker} g))\). Note that
		$$(f,g) \circ (\operatorname{ker}f, \operatorname{ker}g) = (f \circ \operatorname{ker}f, g \circ \operatorname{ker}g) = (0,0) = 0.$$
		Thus, it remains to show the universal property of the kernel. For this, consider $(h, k): (X, Y, H) \rightarrow (A, B, \alpha) \in \operatorname{Mor}((F/G))$ such that $(f, g) \circ (h, k) = (0,0)$. Since $(h, k)$ is a morphism in $(F/G)$, we have
		\begin{equation}
			\alpha \circ F(h) = G(k) \circ H .
			\label{hk_morfi}
		\end{equation}
		Since $f \circ h = 0$ and $g \circ k = 0$, by the universal property of the kernels of $f$ and $g$, there exist unique morphisms $\varphi: X \rightarrow \operatorname{Ker} f$ and $\psi: Y \rightarrow \operatorname{Ker} g$ such that
		\begin{equation}
			h = \operatorname{ker}f \circ \varphi \ \mbox{ and } \ k = \operatorname{ker}g \circ \psi .
			\label{hk}
		\end{equation}
		By \cref{hk} we have
		\begin{equation}
			F(h) = F(\operatorname{ker}f) \circ F(\varphi) \ \mbox{ and } \ G(k) = G(\operatorname{ker}g) \circ G(\psi).
			\label{FhGk}
		\end{equation}
		Furthermore, by \cref{beta_G_exac_esq}
		\[
		G(g) \circ (\alpha \circ  F(\operatorname{ker}f)) \circ F(\varphi) =
		G(g) \circ (G(\operatorname{ker}g) \circ \beta) \circ F(\varphi)  = 0 \circ \beta \circ F(\varphi) = 0 .
		\]
		Thus, by \cref{FhGk} and the universal property of the kernel of $G(g)$, there exists a unique $\Phi: F(X) \rightarrow G(\operatorname{Ker} g)$ such that
		\begin{equation}
			\alpha \circ  F(h) = G(\operatorname{ker}g) \circ \Phi.
			\label{alphPhiFG}
		\end{equation}
		Thus, by \cref{beta_G_exac_esq}, \cref{FhGk} and \cref{alphPhiFG}
		\begin{equation}
			G(\operatorname{ker}g) \circ \beta \circ  F(\varphi) =
			\alpha \circ F(\operatorname{ker}f) \circ F(\varphi) =
			\alpha \circ F(h) =
			G(\operatorname{ker}g) \circ \Phi .
			\label{GalpF}
		\end{equation}
		By the left exactness of $G$, we can rewrite equation \cref{GalpF} as
		\begin{equation}
			\operatorname{ker}G(g) \circ \beta \circ  F(\varphi) =  \operatorname{ker}G(g) \circ \Phi .
		\end{equation}
		Since \(\operatorname{ker}G(g)\) is a monomorphism we conclude that
		\begin{equation}
			\Phi = \beta \circ  F(\varphi).
			\label{Phi_beta}
		\end{equation}
		On the other hand, from \cref{FhGk}, \cref{hk_morfi} and \cref{alphPhiFG}
		\begin{equation}
			G(\operatorname{ker}g) \circ G(\psi) \circ H =
			G(k) \circ H =
			\alpha \circ F(h) =
			G(\operatorname{ker}g) \circ \Phi .
			\label{GGH}
		\end{equation}
		Again, using the left exactness of \(G\) and the fact that \(G(\operatorname{ker}g)\) is a monomorphism, we obtain
		\begin{equation}
			\Phi = G(\psi) \circ  H.
			\label{Phi_H}
		\end{equation}
		From \cref{Phi_beta} and \cref{Phi_H}, we obtain
		\begin{equation}
			\beta \circ  F(\varphi) =  G(\psi) \circ  H,
		\end{equation}
		Thus, the diagram below commutes,
		\begin{equation*}
			\begin{tikzcd}[sep=large]
				F(X) \arrow[r, "F(\varphi)"] \arrow[d, "H"'] & F(\operatorname{Ker}f) \arrow[d, "\beta"] \\ G(Y) \arrow[r, "G(\psi)"']                   & G(\operatorname{Ker}g)                   
			\end{tikzcd}
		\end{equation*}
		and we obtain that $(\varphi, \psi): (X, Y, H) \rightarrow (\operatorname{Ker} f, \operatorname{Ker} g, \beta)$ is a morphism in $(F/G)$. Moreover, by \cref{hk}
		\begin{equation}
			(\operatorname{ker}f, \operatorname{ker}g) \circ (\varphi, \psi) = (\operatorname{ker}f \circ \varphi, \operatorname{ker}g \circ \psi) = (h,k) .
			\label{var_psi_prop_ker}
		\end{equation}
		To conclude, we will show that the morphism $(\varphi, \psi)$ is unique with the property of \cref{var_psi_prop_ker}. Indeed, suppose $(\eta, \rho): (X, Y, H) \rightarrow (\operatorname{Ker} f, \operatorname{Ker} g, \beta)$ is a morphism in $(F/G)$ such that $(\operatorname{ker} f, \operatorname{ker} g) \circ (\eta, \rho) = (h, k)$. Thus,
		$$(h,k) =  (\operatorname{ker}f, \operatorname{ker}g) \circ (\eta, \rho) =  (\operatorname{ker}f \circ \eta, \operatorname{ker}g \circ \rho),$$
		that is, $\operatorname{ker}f \circ \eta = h$ and $\operatorname{ker}g \circ \rho = k$. From the uniqueness of $\varphi$ and $\psi$, we obtain $\eta = \varphi$ and $\rho = \psi$. Therefore, \((F/G)\) admits kernels.
	\end{proof}
	
	\begin{proposition}
		Let $\mathcal{A}$, $\mathcal{B}$, and $\mathcal{C}$ be abelian categories, $F: \mathcal{A} \rightarrow \mathcal{C}$ a right exact functor, and $G: \mathcal{B} \rightarrow \mathcal{C}$ an additive functor. Then, $(F/G)$ admits cokernels.
		\label{prop_comma_cokernel}
	\end{proposition}
	\begin{proof}
		The argument parallels that of \cref{prop_comma_kernel}, by swapping the role of the kernel with the cokernel. For any morphism $(f,g): (A,B,\alpha) \rightarrow (A',B',\alpha')$ in $(F/G)$, we have that the cokernel of $(f,g)$ is given by 
		\[\left(\operatorname{Coker}(f,g), \operatorname{coker}(f,g)\right) = \left((\operatorname{Coker} f, \operatorname{Coker} g, \gamma), (\operatorname{coker} f, \operatorname{coker} g)\right),\]
		where $\gamma: F(\operatorname{Coker} f) \rightarrow G(\operatorname{Coker} g)$ is the unique morphism making the diagram below commute.
		\begin{equation*}
			\begin{tikzcd}[sep=large]
				F(A) \arrow[r, "F(f)"] & F(A') \arrow[r, "F(\operatorname{coker}f)"] \arrow[d, "G(\operatorname{coker}g) \circ \alpha'"'] & F(\operatorname{Coker}f) \arrow[ld, "\gamma", dashed] \\ & G(\operatorname{Coker}g)   &           \end{tikzcd}
		\end{equation*}
	\end{proof}

	\begin{theorem}
		Let $\mathcal{A}$, $\mathcal{B}$, and $\mathcal{C}$ be abelian categories, $F: \mathcal{A} \rightarrow \mathcal{C}$ a right exact functor, and $G: \mathcal{B} \rightarrow \mathcal{C}$ a left exact functor. Then, $(F/G)$ is an abelian category.
		\label{teo_comma_abeliana}
	\end{theorem}
	\begin{proof}
		By \cref{prop_comma_aditiva}, \cref{prop_comma_kernel}, and \cref{prop_comma_cokernel}, it remains to show that, for every morphism in \((F/G)\), its induced morphism is an isomorphism.
		Let $(f,g): (A,B,\alpha) \rightarrow (A',B',\alpha')$ be any morphism in $(F/G)$. By the basic property of categories with kernels and cokernels, there exist unique morphisms \(\varphi: A \rightarrow \operatorname{Im}f\), \( \overline{f}: \operatorname{Coim}f \rightarrow \operatorname{Im}f\) in \(\mathcal{A}\) and \(\psi: B \rightarrow \operatorname{Im}g\), \( \overline{g}: \operatorname{Coim}g \rightarrow \operatorname{Im}g\) in \(\mathcal{B}\) such that
		\begin{equation}
			\label{eq_morf_ind_cat_A}
			f = \beta \circ \varphi \ \mbox{ and } \ \varphi = \overline{f} \circ \theta
		\end{equation}
		\begin{equation}
			\label{eq_morf_ind_cat_B}
			g = \Tilde{\beta} \circ \psi \ \mbox{ and } \ \psi = \overline{g} \circ \Tilde{\theta}
		\end{equation}
		where \(\theta = \operatorname{coker}(\operatorname{ker}f)\), \(\beta = \operatorname{ker}(\operatorname{coker}f)\), \(\Tilde{\theta} = \operatorname{coker}(\operatorname{ker}g)\) and \(\Tilde{\beta} = \operatorname{ker}(\operatorname{coker}g)\).
		Note that
		\begin{eqnarray*}
			\operatorname{Im}(f,g) = \operatorname{Ker}(\operatorname{coker}(f,g)) 
			&=&
			\operatorname{Ker}(\operatorname{coker}f,\operatorname{coker}g) \\
			&=&
			(\operatorname{Ker}(\operatorname{coker}f),\operatorname{Ker}(\operatorname{coker}g), \Delta) \\
			&=&
			(\operatorname{Im}f, \operatorname{Im}g, \Delta), 
		\end{eqnarray*}
		where \(\Delta: F(\operatorname{Im}f) \rightarrow G(\operatorname{Im}g)\) arises from the universal property of the kernel of \(G(\operatorname{coker}g)\).
		\begin{eqnarray*}
			\operatorname{Coim}(f,g) = \operatorname{Coker}(\operatorname{ker}(f,g)) 
			&=&\operatorname{Coker}(\operatorname{ker}f,\operatorname{ker}g) \\
			&=&
			(\operatorname{Coker}(\operatorname{ker}f),\operatorname{Coker}(\operatorname{ker}g), \Gamma) \\
			&=&
			(\operatorname{Coim}f, \operatorname{Coim}g, \Gamma),
		\end{eqnarray*}
		where \(\Gamma: F(\operatorname{Coim}f) \rightarrow G(\operatorname{Coim}g)\) arises from the universal property of the cokernel of \(F(\operatorname{ker}f)\). First we will prove that \((\varphi, \psi): (A,B,\alpha) \rightarrow (\operatorname{Im}f, \operatorname{Im}g, \Delta)\) is a morphism in \((F/G)\). Note that
		\[
		\operatorname{coker}(\operatorname{ker}(f,g)) = \operatorname{coker}(\operatorname{ker}f, \operatorname{ker}g) = (\operatorname{coker}(\operatorname{ker}f), \operatorname{coker}(\operatorname{ker}g)) = (\theta, \Tilde{\theta})
		\]
		\[
		\operatorname{ker}(\operatorname{coker}(f,g)) = \operatorname{ker}(\operatorname{coker}f, \operatorname{coker}g) = (\operatorname{ker}(\operatorname{coker}f), \operatorname{ker}(\operatorname{coker}g)) = (\beta, \Tilde{\beta})
		\]
		are morphisms in \((F/G)\). Then the following diagram commutes:
		\begin{equation*}
			\begin{tikzcd}
				F(A) \arrow[r, "F(\theta)"] \arrow[d, "\alpha"'] & F(\operatorname{Coim}f) \arrow[d, "\Gamma"] &  & F(\operatorname{Im}f) \arrow[r, "F(\beta)"] \arrow[d, "\Delta"'] & F(A') \arrow[d, "\alpha'"] \\
				G(B) \arrow[r, "G(\Tilde{\theta})"']             & G(\operatorname{Coim}g)                     &  & G(\operatorname{Im}g) \arrow[r, "G(\Tilde{\beta})"']            & G(B')                     
			\end{tikzcd}
		\end{equation*}
		From the diagram e using that \((f,g)\) is a morphism in \((F/G)\) we have
		\begin{equation}
			\label{eq_morf_fg_comma}
			\alpha' \circ F(f) = G(g) \circ \alpha.
		\end{equation}
		\begin{equation}
			\label{eq_morf_coim_comma}
			\Gamma \circ F(\theta) = G(\Tilde{\theta}) \circ \alpha.
		\end{equation}
		\begin{equation}
			\label{eq_morf_im_comma}
			\alpha' \circ F(\beta) = G(\Tilde{\beta}) \circ \Delta.
		\end{equation}
		From \cref{eq_morf_ind_cat_A} and \cref{eq_morf_coim_comma}, we have
		\[
		\alpha' \circ F(f) = \alpha' \circ F(\beta \circ \varphi) =  \alpha' \circ F(\beta) \circ F(\varphi) = G(\Tilde{\beta}) \circ \Delta \circ F(\varphi).
		\]
		By \cref{eq_morf_ind_cat_B}:
		\[
		G(g) \circ \alpha = G(\Tilde{\beta} \circ \psi) \circ \alpha = G(\Tilde{\beta}) \circ G(\psi) \circ \alpha.
		\]
		From \cref{eq_morf_fg_comma}, it follows that
		\[
		G(\Tilde{\beta}) \circ \Delta \circ F(\varphi)= G(\Tilde{\beta}) \circ G(\psi) \circ \alpha.
		\]
		Since \(\tilde{\beta}\) is a monomorphism, it follows that \(\tilde{\beta} = \operatorname{ker}(\operatorname{coker} \tilde{\beta})\). Moreover, since \(G\) is a left exact functor, we have:
		\begin{equation}
			\label{eq_morf_varphi_psi_comma}
			\Delta \circ F(\varphi)=  G(\psi) \circ \alpha,
		\end{equation}
		thus \((\varphi, \psi)\) is a morphism in \((F/G)\).
		
		Now we will prove that \((\overline{f}, \overline{g}) : \operatorname{Coim}(f,g) \rightarrow \operatorname{Im}(f,g)\) is a morphism in \((F/G)\). From \cref{eq_morf_ind_cat_A}:
		\[
		\Delta \circ F(\varphi) = \Delta \circ F(\overline{f}) \circ F(\theta).
		\]
		By \cref{eq_morf_ind_cat_B} and \cref{eq_morf_coim_comma}:
		\[
		G(\psi) \circ \alpha = G(\overline{g}) \circ G(\Tilde{\theta}) \circ \alpha = G(\overline{g}) \circ\Gamma \circ F(\theta).
		\]
		From \cref{eq_morf_varphi_psi_comma}, we have
		\[
		\Delta \circ F(\overline{f}) \circ F(\theta) =  G(\overline{g}) \circ \Gamma \circ F(\theta).
		\]
		Since \(\theta\) is an epimorphism, it follows that \(\theta = \operatorname{coker}(\operatorname{ker} \theta)\). Moreover, since \(F\) is a right exact functor, we have:
		\begin{equation}
			\label{eq_mor_ind_comma_cand}
			\Delta \circ F(\overline{f}) = G(\overline{g}) \circ \Gamma,
		\end{equation}
		then \((\overline{f},\overline{g})\) is a morphism in \((F/G)\). Moreover, from \cref{eq_morf_ind_cat_A} and \cref{eq_morf_ind_cat_B}
		\[
		(f,g) = (\beta, \Tilde{\beta}) \circ (\overline{f}, \overline{g}) \circ (\theta, \Tilde{\theta}) = \operatorname{ker}(\operatorname{coker}(f,g)) \circ (\overline{f}, \overline{g}) \circ
		\operatorname{coker}(\operatorname{ker}(f,g)).
		\]
		Thus \((\overline{f}, \overline{g})\) is the induced morphism of \((f,g)\). 
		
		Finally we will prove that \((\overline{f}, \overline{g})\) is an isomorphism. Since \(\mathcal{A}\) and \(\mathcal{B}\) are abelian we have \(\overline{f}\) and \(\overline{g}\) isomorphisms. Thus, composing \(G((\overline{g})^{-1})\) on the left and \(F((\overline{f})^{-1})\) on the right in \cref{eq_mor_ind_comma_cand}, we obtain
		\[
		G((\overline{g})^{-1}) \circ \Delta = \Gamma \circ  F((\overline{f})^{-1}),
		\]
		thus \(((\overline{f})^{-1}, (\overline{g})^{-1})\) is a morphism in \((F/G)\). Besides that,
		\[
		(\overline{f}, \overline{g}) \circ ((\overline{f})^{-1},(\overline{g})^{-1}) = (\overline{f} \circ (\overline{f})^{-1}, \overline{g} \circ (\overline{g})^{-1}) = (1_{\operatorname{Im}f}, 1_{\operatorname{Im}g}) = 1_{\operatorname{Im}(f,g)} .
		\]
		\[
		((\overline{f})^{-1},(\overline{g})^{-1}) \circ (\overline{f}, \overline{g}) = ((\overline{f})^{-1} \circ \overline{f}, (\overline{g})^{-1} \circ \overline{g}) = (1_{\operatorname{Coim}f}, 1_{\operatorname{Coim}g}) = 1_{\operatorname{Coim}(f,g)} .
		\]
		Hence, the induced morphism of \((f,g)\) is an isomorphism. Therefore, \((F/G)\) is an abelian category.
	\end{proof}
	
	\begin{remark}
		\label{obs_contraexemplo}
		The converse of this theorem does not hold. Let \(\mathcal{A} = \mathcal{B} = \mathcal{C} = \operatorname{Ab}\) be the category of abelian groups, \(F: \mathcal{A} \rightarrow \mathcal{C}\) be the functor \(F = \mathbb{Z} \oplus -\), i.e., \(F(A) = \mathbb{Z} \oplus A\) and \(F(f) = 1_{\mathbb{Z}} \oplus f\); and \(G:  \mathcal{B} \rightarrow \mathcal{C}\) be the trivial functor \(G(B) = 0\) and \(G(g)=0\). 
		We first show that \(F\) fails to be right exact. Indeed, consider the short exact sequence
		\[
		0 \rightarrow \mathbb{Z} \xrightarrow{\varphi} \mathbb{Z} \xrightarrow{\psi} \mathbb{Z}_{2} \rightarrow 0,
		\]
		where \(\varphi(n) = 2n\) and \(\psi(n) = n\operatorname{mod}2\). Note that \(F(\varphi): \mathbb{Z} \oplus \mathbb{Z} \rightarrow \mathbb{Z} \oplus \mathbb{Z}\) is such that \(F(\varphi)(x,y) = (x,2y) \), this way we have that \(\operatorname{Im}F(\varphi) = \mathbb{Z} \oplus 2\mathbb{Z}\). Besides that, we have \(F(\psi): \mathbb{Z} \oplus \mathbb{Z} \rightarrow \mathbb{Z} \oplus \mathbb{Z}_{2}\) such that \(F(\psi)(x,y) = (x,y\operatorname{mod}2) \), thus \(\operatorname{Ker}F(\psi) = 0 \oplus 2\mathbb{Z}\). Thus, \(\operatorname{Im}F(\varphi) \neq \operatorname{Ker}F(\psi)\), so the sequence \( F(\mathbb{Z}) \xrightarrow{F(\varphi)} F(\mathbb{Z}) \xrightarrow{F(\psi)} F(\mathbb{Z}_{2}) \rightarrow 0\) is not exact, i.e., \(F\) is not a right exact functor. Now we will prove that \((F/G)\) is abelian. The objects are of the form \((A,B,0)\), because \(G(B) = 0 \) for all \(B \in  \mathcal{B}\). Moreover, for every \(f: A \rightarrow A' \in \operatorname{Mor}(\mathcal{A})\) and \(g: B \rightarrow B' \in \operatorname{Mor}(\mathcal{B})\) we have the diagram
		\begin{equation*}
			\begin{tikzcd}
				F(A) \arrow[r, "F(f)"] \arrow[d, "0"'] & F(A') \arrow[d, "0"] \\
				G(B) \arrow[r, "G(g)"']                & G(B')               
			\end{tikzcd}
		\end{equation*}
		is commutative. Thus, given morphisms \(f \in \operatorname{Mor}(\mathcal{A})\) and \(g \in \operatorname{Mor}(\mathcal{B})\) we see that \((f,g)\) is a morphism in \((F/G)\). The functor \(\Psi: (F/G) \rightarrow \mathcal{A} \times \mathcal{B}\) given by 
		\begin{eqnarray*}
			\Psi((A,B,0)) &=&  (A,B) \\
			\Psi((f,g)) &=& (f,g)
		\end{eqnarray*}
		is an equivalence of categories, because \(\Psi\) is full, faithful, and essentially surjective. This way \((F/G) \cong \mathcal{A} \times \mathcal{B} = \operatorname{Ab} \times \operatorname{Ab}\) is abelian, because the product category of the abelian categories is abelian. Hence, \((F/G)\) is abelian although the functor \(F\) fails to be right exact, showing that the converse of \cref{teo_comma_abeliana} does not hold.
	\end{remark}
	
	\section{Grothendieck Group of Comma Categories}
	\label{sect_GG}
	\subsection{Short Exact Sequences}  
	In this section, assume the conditions of \cref{teo_comma_abeliana} hold. We explore the concepts of monomorphism, epimorphism, isomorphism, and exact sequences in the context of comma categories. We analyze the conditions under which these notions in a comma category induce corresponding properties in the original categories and vice versa. Although these results are simple, they play a fundamental role throughout the text.
	
	\begin{lemma}
		Let \( (f,g): (A,B,\alpha) \rightarrow (A',B',\alpha') \) be any morphism in \( (F/G) \).
		\begin{enumerate}
			\item If $(f,g)$ is a monomorphism in $(F/G)$, then $f$ is a monomorphism in $\mathcal{A}$ and $g$ is a monomorphism in $\mathcal{B}$;
			
			\item If $(f,g)$ is an epimorphism in $(F/G)$, then $f$ is an epimorphism in $\mathcal{A}$ and $g$ is an epimorphism in $\mathcal{B}$.
		\end{enumerate}
		\label{mono_epi_comma}
	\end{lemma}
	\begin{proof}
		\noindent
		\begin{enumerate}
			\item Suppose $(f,g)$ is a monomorphism in $(F/G)$. Then, we have \((0_{\mathcal{A}},0_{\mathcal{B}},0) = \operatorname{Ker}(f,g) = (\operatorname{Ker} f, \operatorname{Ker} g, \beta)\). Thus, $\operatorname{Ker} f = 0_{\mathcal{A}}$ and $\operatorname{Ker} g = 0_{\mathcal{B}}$. Then, we have that $f$ is a monomorphism in $\mathcal{A}$ and $g$ is a monomorphism in $\mathcal{B}$.
			
			\item Suppose that $(f,g)$ is an epimorphism in $(F/G)$. Then we have  $(0_{\mathcal{A}},0_{\mathcal{B}},0) = \operatorname{Coker}(f,g) = (\operatorname{Coker} f, \operatorname{Coker} g, \gamma)$. Thus, $\operatorname{Coker} f = 0_{\mathcal{A}}$ and $\operatorname{Coker} g = 0_{\mathcal{B}}$. Hence, it follows that  $f$ is an epimorphism in $\mathcal{A}$ and $g$ is an epimorphism in $\mathcal{B}$.
		\end{enumerate}
	\end{proof}
	
	\begin{corollary}
		If \( (f,g): (A,B,\alpha) \rightarrow (A',B',\alpha') \) is an isomorphism in \( (F/G) \) then \(f:A \rightarrow A'\) is an isomorphism in \(\mathcal{A}\) and  \(g:B \rightarrow B'\) is an isomorphism in \(\mathcal{B}\).
		\label{cor_iso_comma_imp_iso_A_B}
	\end{corollary}
	\begin{proof}
		Since \((f,g)\) is an isomorphism in \((F/G)\), it follows that \((f,g)\) is both a monomorphism and an epimorphism. By \cref{mono_epi_comma}, we know that \( f \) is a monomorphism and an epimorphism in the abelian category \( \mathcal{A} \), and \( g \) is a monomorphism and an epimorphism in the abelian category \( \mathcal{B} \). Hence, we have \( f \) an isomorphism in \( \mathcal{A} \) and \( g \) an isomorphism in \( \mathcal{B} \).
	\end{proof}
	
	\begin{corollary}
		If
		\begin{equation}
			\mathcal{\mathbf{S}} = 0 \rightarrow (A', B', \alpha') \xrightarrow{(f,g)}(A,B,\alpha)  \xrightarrow{(h,k)} (A'', B'', \alpha'') \rightarrow 0
			\label{seq_S}
		\end{equation}
		is a short exact sequence in $(F/G)$ then $0 \rightarrow A' \xrightarrow{f} A \xrightarrow{h} A'' \rightarrow 0$ is a short exact sequence in $\mathcal{A}$, and $0 \rightarrow B' \xrightarrow{g} B \xrightarrow{k} B'' \rightarrow 0$ is a short exact sequence in $\mathcal{B}$.
		\label{seq_exat_curt_comma}
	\end{corollary}
	\begin{proof}
		By hypothesis, $(f,g)$ is a monomorphism in $(F/G)$, so by  \cref{mono_epi_comma}, we have that $f$ is a monomorphism in $\mathcal{A}$ and $g$ is a monomorphism in $\mathcal{B}$. Furthermore, $(h,k)$ is an epimorphism in $(F/G)$, and again by \cref{mono_epi_comma}, we obtain that $h$ is an epimorphism in $\mathcal{A}$ and $k$ is an epimorphism in $\mathcal{B}$. Finally,
		$\operatorname{Im}(f,g) = \operatorname{Ker}(h,k)$, that is, $(\operatorname{Im} f, \operatorname{Im} g, \beta) = (\operatorname{Ker} h, \operatorname{Ker} k, \tilde{\beta})$, so $\operatorname{Im} f = \operatorname{Ker} h$ and $\operatorname{Im} g = \operatorname{Ker} k$. Thus, $0 \rightarrow A' \xrightarrow{f} A \xrightarrow{h} A'' \rightarrow 0$ is a short exact sequence in $\mathcal{A}$, and $0 \rightarrow B' \xrightarrow{g} B \xrightarrow{k} B'' \rightarrow 0$ is a short exact sequence in $\mathcal{B}$.
	\end{proof}
	
	\begin{remark}
		Note that a short exact sequence in $(F/G)$ does not induce a short exact sequence in $\mathcal{C}$. Indeed, given a short exact sequence $\mathcal{\mathbf{S}}$ in $(F/G)$, as in \cref{seq_S}, by \cref{seq_exat_curt_comma}, we have that $0 \rightarrow A' \xrightarrow{f} A \xrightarrow{h} A'' \rightarrow 0$ is a short exact sequence in $\mathcal{A}$, and $0 \rightarrow B' \xrightarrow{g} B \xrightarrow{k} B'' \rightarrow 0$ is a short exact sequence in $\mathcal{B}$. From the right exactness of \(F\) we have $F(A') \xrightarrow{F(f)} F(A) \xrightarrow{F(h)} F(A'') \rightarrow 0$ an exact sequence in $\mathcal{C}$, and from the left exactness of \(G\) we have $0 \rightarrow G(B') \xrightarrow{G(g)} G(B) \xrightarrow{G(k)} G(B'')$ an exact sequence in $\mathcal{C}$. 
		\label{obs_seq_exac_curt_comma_NAO_seq-exat_curt_C}
	\end{remark}
	
	\begin{remark}
		By \cref{obs_seq_exac_curt_comma_NAO_seq-exat_curt_C} a short exact sequence $\mathcal{\mathbf{S}}$ in $(F/G)$, as in \cref{seq_S}, has the form
		\begin{equation*}
			\begin{tikzcd}
				& F(A') \arrow[r, "F(f)"] \arrow[d, "\alpha'"'] & F(A) \arrow[r, "F(h)"] \arrow[d, "\alpha"] & F(A'') \arrow[r] \arrow[d, "\alpha''"] & 0 \\ 0 \arrow[r] & G(B') \arrow[r, "G(g)"'] & G(B) \arrow[r, "G(k)"']  & G(B'') &  
			\end{tikzcd}
		\end{equation*}
		\label{obs_seq_exat_curta-comma}
	\end{remark}
	
	\begin{lemma}
		Let \((f,g): (A,B,\alpha) \rightarrow (A',B',\alpha')\) be a morphism in \((F/G)\). 
		\begin{enumerate}
			\item If \(f\) is a monomorphism in \(\mathcal{A}\) and \(g\) is a monomorphism in \(\mathcal{B}\) then \((f,g)\) is a monomorphism in \((F/G)\);
			
			\item If \(f\) is an epimorphism in \(\mathcal{A}\) and \(g\) is an epimorphism in \(\mathcal{B}\) then \((f,g)\) is an epimorphism in \((F/G)\).
		\end{enumerate}
		\label{lem_morfi_sist_mono_cat_ini_mono_sist}
	\end{lemma}
	\begin{proof}
		\noindent
		\begin{enumerate}
			\item Since \(f\) and \(g\) are monomorphisms in \(\mathcal{A}\) and \(\mathcal{B}\) respectively, we have \(\operatorname{Ker}f = 0_\mathcal{A} \) and \(\operatorname{Ker}g = 0_\mathcal{B}\). Then
			\[\operatorname{Ker}(f,g) = (\operatorname{Ker}f, \operatorname{Ker}g,\beta) = (0_\mathcal{A},0_\mathcal{B},\beta) =(0_\mathcal{A},0_\mathcal{B},0).\] 
			Hence, it follows that \((f,g)\) is a monomorphism.
			
			\item Since \(f\) and \(g\) are epimorphisms in \(\mathcal{A}\) and \(\mathcal{B}\) respectively, we have  \(\operatorname{Coker}f = 0_\mathcal{A}\) and \(\operatorname{Coker}g = 0_\mathcal{B}\). Then \[\operatorname{Coker}(f,g) = (\operatorname{Coker}f, \operatorname{Coker}g,\gamma) = (0_\mathcal{A},0_\mathcal{B},\gamma) =(0_\mathcal{A},0_\mathcal{B},0).\] 
			Hence, it follows that \((f,g)\) is an epimorphism.
		\end{enumerate}
	\end{proof}
	
	\begin{corollary}
		Let \( (f,g): (A,B,\alpha) \rightarrow (A',B',\alpha') \) be a morphism in \( (F/G) \).  If \(f:A \rightarrow A'\) is an isomorphism in \(\mathcal{A}\) and  \(g:B \rightarrow B'\) is an isomorphism in \(\mathcal{B}\) then \((f,g)\) is an isomorphism in  \( (F/G) \).
		\label{cor_iso_A_B_imp_iso_comma}
	\end{corollary}
	\begin{proof}
		Given that \(f\) and \(g\) are isomorphisms in \( \mathcal{A} \) and \( \mathcal{B} \), respectively, they are both monomorphisms and epimorphisms in \( \mathcal{A} \) and \( \mathcal{B} \), respectively. By \cref{lem_morfi_sist_mono_cat_ini_mono_sist} we conclude that \((f,g)\) is both a monomorphism and an epimorphism in the abelian category \((F/G)\), therefore it follows that \((f,g)\) is an isomorphism in \((F/G)\).
	\end{proof}
	
	\subsection{Grothendieck Group}  
	In this section, we analyze how the Grothendieck group of a comma category relates to the Grothendieck groups of the original categories.
	By \cref{teo_comma_abeliana}, we have that $(F/G)$ is an abelian category, and thus we can talk about the Grothendieck group of $(F/G)$.
	
	\begin{remark}
		\label{obs_GG_comma_nao_depende_GG_C}
		For every object \((A,B,\alpha)\) in \((F/G)\) we have \([(A,B,\alpha)] = [(A,B,0)]\) in \(K((F/G))\).
		
		In fact, let $(A,B,\alpha)$ be an arbitrary object in $(F/G)$. Note that the diagram below commutes.
		\begin{equation*}
			\begin{tikzcd}
				& 0 \arrow[r, "0"] \arrow[d, "0"'] & F(A) \arrow[r, "F(1_{A})"] \arrow[d, "\alpha"] & F(A) \arrow[r] \arrow[d, "0"] & 0 \\
				0 \arrow[r] & G(B) \arrow[r, "G(1_{B})"']      & G(B) \arrow[r, "0"']    & 0  &  
			\end{tikzcd}
		\end{equation*}
		Thus, by \cref{obs_seq_exat_curta-comma}, 
		$$0 \rightarrow (0, B, 0) \xrightarrow{(0,1_{B})}(A,B,\alpha)  \xrightarrow{(1_{A},0)} (A, 0, 0) \rightarrow 0$$
		is a short exact sequence in $(F/G)$. This leads to the equality
		\begin{equation}
			[(A,B,\alpha)] = [(0, B, 0)] + [(A, 0, 0) ] 
			\label{GG_alf}
		\end{equation}
		in $K((F/G))$. For \(\alpha = 0\) we have
		\begin{equation}
			[(A,B,0)] = [(0, B, 0)] + [(A, 0, 0) ] 
			\label{GG_0}
		\end{equation}
		in $K((F/G))$.
		Thus, by \cref{GG_alf} and \cref{GG_0}, we have
		\begin{equation*}
			[(A,B,\alpha)] = [(A,B,0)]
		\end{equation*}
		in $K((F/G))$.
	\end{remark}
	
	Define the full subcategories \( \mathcal{\Tilde{A}} \) and \( \mathcal{\Tilde{B}} \) in \( (F/G) \), whose objects are:
	$$\operatorname{Obj}\mathcal{\Tilde{A}} = \{(A,0,0);\ A \in \mathcal{A}\} \ \ \ \mbox{ and } \ \ \ \operatorname{Obj}\mathcal{\Tilde{B}} = \{(0,B,0);\ B \in \mathcal{B}\}$$
	and whose morphisms coincide with the morphisms in $(F/G)$. 
	
	\begin{lemma}
		The categories \( \mathcal{\Tilde{A}} \) and \( \mathcal{\Tilde{B}} \) are abelian categories.
		\label{lem_sub_plen_sist}
	\end{lemma}
	\begin{proof}
		We have seen that a zero object in \( (F/G) \) is of the form \( (0_{\mathcal{A}}, 0_{\mathcal{B}},0) \), where \( 0_{\mathcal{A}} \in \mathcal{A} \) and \( 0_{\mathcal{B}} \in \mathcal{B} \). Hence, the zero object of \( (F/G) \) belongs to both \( \mathcal{\Tilde{A}} \) and \( \mathcal{\Tilde{B}} \).  
		
		Given objects \( (A,0,0) \) and \( (A',0,0) \) in \( \mathcal{A} \), consider \( (A \oplus A', i_{1}^{A}, i_{2}^{A}) \) as the coproduct of \( A \) and \( A' \). Based on the previous results, the coproduct of \( (A,0,0) \) and \( (A',0,0) \) is given by  
		\[
		((A \oplus A',0,0), (i_{1}^{A},0),(i_{2}^{A},0)),
		\]
		with \((A \oplus A',0,0) \in \operatorname{Obj}(\mathcal{\Tilde{A}})\). Similarly, the coproduct between \( (0,B,0) \) and \( (0,B',0) \), for any objects \( B,B' \in \mathcal{B} \), is given by \(((0,B \oplus B',0), (0,i_{1}^{B}),(0,i_{2}^{B}))\), with \((0,B \oplus B',0) \in  \mathcal{\Tilde{B}} \).  
		
		Let \( (f,0) \) be a morphism in \( \mathcal{\Tilde{A}} \), by our previous calculations, we have 
		\[
		\operatorname{Ker}(f,0) = (\operatorname{Ker}f, 0, 0) \in \operatorname{Obj}(\mathcal{\Tilde{A}}) \mbox{ and } \operatorname{ker}(f,0) = (\operatorname{ker}f, 0) \in \operatorname{Mor}(\mathcal{\Tilde{A}}).
		\]
		\[
		\operatorname{Coker}(f,0) = (\operatorname{Coker}f, 0, 0) \in \operatorname{Obj}(\mathcal{\Tilde{A}}) \mbox{ and } \operatorname{coker}(f,0) = (\operatorname{coker}f, 0) \in \operatorname{Mor}(\mathcal{\Tilde{A}}).
		\]
		Similarly, let \( (0,g) \) be a morphism in \( \mathcal{\Tilde{B}} \), we obtain  
		\[
		\operatorname{Ker}(0,g) = (0, \operatorname{Ker}g,0) \in \operatorname{Obj}(\mathcal{\Tilde{B}}) \mbox{ and }
		\operatorname{ker}(0,g) = (0,\operatorname{ker}g) \in \operatorname{Mor}(\mathcal{\Tilde{B}}),
		\]
		\[
		\operatorname{Coker}(0,g) = (0,\operatorname{Coker}g,0) \in \operatorname{Obj}(\mathcal{\Tilde{B}}) \mbox{ and }
		\operatorname{coker}(0,g) = (0,\operatorname{coker}g) \in \operatorname{Mor}(\mathcal{\Tilde{B}}).
		\]
		Thus, by Proposition 5.92 of the \cite{rotman2009introduction} we have to \( \mathcal{\Tilde{A}} \) and \( \mathcal{\Tilde{B}} \) are abelian categories.
	\end{proof}

	\begin{proposition}
		In the previous notation,  $K((F/G)) \cong K(\mathcal{\Tilde{A}}) \oplus  K(\mathcal{\Tilde{B}})$.
		\label{prop_GG_subcat_plen}
	\end{proposition}
	\begin{proof}
		Define $\eta: \operatorname{Obj}((F/G)) \rightarrow K(\mathcal{\Tilde{A}}) \oplus  K(\mathcal{\Tilde{B}})$ such that
		\begin{equation*}
			\eta((A,B,\alpha)) = [(A,0,0)] \oplus [(0,B,0)].
		\end{equation*}
		First, we show that the function \( \eta \) is additive. Indeed, consider
		\[
		0 \rightarrow (A', B', \alpha') \rightarrow (A, B, \alpha) \rightarrow (A'', B'', \alpha'') \rightarrow 0
		\]  
		any short exact sequence in \( (F/G) \). 
		In this case, we obtain the following diagram:
		$$  
		\begin{tikzcd}
			& F(A') \arrow[r] \arrow[d, "\alpha'"'] & F(A) \arrow[r] \arrow[d, "\alpha"'] & F(A'') \arrow[r] \arrow[d, "\alpha''"] & 0 \\
			0 \arrow[r] & G(B') \arrow[r]                       & G(B) \arrow[r]                      & G(B'')                                 &  
		\end{tikzcd}
		$$
		that induces
		$$
		\begin{tikzcd}
			& F(A') \arrow[r] \arrow[d, "0"'] & F(A) \arrow[r] \arrow[d, "0"'] & F(A'') \arrow[r] \arrow[d, "0"] & 0 \\ 0 \arrow[r] & 0 \arrow[r]        & 0 \arrow[r]  & 0  &  
		\end{tikzcd}
		$$
		showing that the sequence
		\begin{equation*}
			0 \rightarrow (A',0,0) \rightarrow (A,0,0) \rightarrow (A'',0,0) \rightarrow 0
		\end{equation*}
		is short exact in $\mathcal{\Tilde{A}}$. Then
		\begin{equation}
			[(A,0,0)] = [(A',0,0)] + [(A'',0,0)]
			\label{K(Atil)}
		\end{equation}
		in $K(\mathcal{\Tilde{A}})$. Moreover, we have the following diagram:
		$$
		\begin{tikzcd}
			& 0 \arrow[r] \arrow[d, "0"'] & 0 \arrow[r] \arrow[d, "0"] & 0 \arrow[r] \arrow[d, "0"] & 0 \\ 0 \arrow[r] & G(B') \arrow[r]             & G(B) \arrow[r] & G(B'')  &  
		\end{tikzcd}
		$$
		which tells us that the sequence 
		\begin{equation*}
			0 \rightarrow (0,B',0) \rightarrow (0,B,0) \rightarrow (0,B'',0) \rightarrow 0
		\end{equation*}
		is a short exact sequence in \( \mathcal{\Tilde{B}} \). Thus,
		\begin{equation}
			[(0,B,0)] = [(0,B',0)] + [(0,B'',0)]
			\label{K(Btil)}
		\end{equation}
		in $K(\mathcal{\Tilde{B}})$. 
		And so, by the \cref{K(Atil)} and \cref{K(Btil)}
		\begin{eqnarray*}
			\eta((A,B,\alpha)) 
			& = &
			[(A,0,0)] \oplus [(0,B,0)] \\
			& = &
			\left( [(A',0,0)] + [(A'',0,0)]) \oplus  ([(0,B',0)] + [(0,B'',0)] \right) \\
			& = &
			[(A',0,0)] \oplus [(0,B',0)]  + [(A'',0,0)] \oplus [(0,B'',0)]\\
			& = &
			\eta((A',B',\alpha'))  + \eta((A'',B'',\alpha'')) .
		\end{eqnarray*}
		This shows that \( \eta \) is additive in short exact sequences of \( (F/G) \) and, by the universal property of the Grothendieck group, induces a group homomorphism  
		\[
		\overline{\eta} : K((F/G)) \rightarrow K(\mathcal{\Tilde{A}}) \oplus  K(\mathcal{\Tilde{B}})
		\]  
		such that  
		\[
		\overline{\eta}\left( [(A,B,\alpha)]\right) =  [(A,0,0)] \oplus [(0,B,0)].
		\]  
		We will show that \( \overline{\eta} \) is a group isomorphism. To prove injectivity, assume
		\[
		\overline{\eta}\left( [(A,B,\alpha)]\right) = 0.
		\]  
		Thus,  
		\[
		\overline{\eta}\left( [(A,B,\alpha)] \right) = 0
		\Rightarrow
		[(A,0,0)] \oplus [(0,B,0)] = 0 
		\Rightarrow
		[(A,0,0)] = 0 = [(0,B,0)].
		\]  
		By \cref{obs_GG_comma_nao_depende_GG_C} and by the equations above, we have 
		\[
		[(A,B,\alpha)] = [(A,B,0)] =
		[(A,0,0)] + [(0,B,0)] = 0 + 0 = 0.
		\]  
		Consequently, \([(A,B,\alpha)] = 0\) in \(K((F/G))\), thus \( \overline{\eta} \) is injective.  
		
		To prove surjectivity, let
		\( [(Z,0,0)] \oplus [(0,W,0)] \in K(\mathcal{\Tilde{A}}) \oplus  K(\mathcal{\Tilde{B}}) \) arbitrary. In this case, take \([(Z,W,0)] \in K((F/G))\). Then
		\[
		\overline{\eta}([(Z,W,0)]) =
		[(Z,0,0)] \oplus [(0,W,0)].
		\]  
		This shows the surjectivity of \( \overline{\eta} \). Therefore, \( \overline{\eta} \) is a group isomorphism. Hence, \(K((F/G)) \cong K(\mathcal{\Tilde{A}}) \oplus  K(\mathcal{\Tilde{B}})\).  
	\end{proof}
	\begin{proposition}
		\label{prop_iso_GG-til}
		In the previous notation,  \( K(\mathcal{\Tilde{A}}) \cong K(\mathcal{A}) \) and \( K(\mathcal{\Tilde{B}}) \cong K(\mathcal{B})\). 
	\end{proposition}
	
	\begin{proof}
		We will show that \( K(\mathcal{\Tilde{A}}) \cong K(\mathcal{A}) \). Define \( \varphi: \operatorname{Obj}(\Tilde{\mathcal{A}}) \rightarrow K(\mathcal{A}) \) such that
		$$\varphi((A,0,0)) = [A].$$
		First, we will show that the function \(  \varphi \) is additive. Indeed, consider
		\begin{equation*}
			0 \rightarrow (A',0,0) \rightarrow (A,0,0) \rightarrow (A'',0,0) \rightarrow 0
		\end{equation*}
		a short exact sequence in \( \Tilde{\mathcal{A}} \).  By \cref{seq_exat_curt_comma}, we know that the sequence
		\begin{equation*}
			0 \rightarrow A' \rightarrow A \rightarrow A'' \rightarrow 0
		\end{equation*}
		is short exact in \( \mathcal{A} \). Hence,
		\begin{equation*}
			[A] = [A'] + [A'']
		\end{equation*}
		in \( K(\mathcal{A}) \).
		Thus,
		\begin{equation*}
			\varphi((A,0,0)) = [A] = [A'] + [A''] = \varphi((A',0,0)) + \varphi((A'',0,0)),
		\end{equation*}
		meaning that \( \varphi \) is additive in short exact sequences of \( \Tilde{\mathcal{A}} \). By the universal property of the Grothendieck group, \( \varphi \) induces a group homomorphism \( \overline{\varphi}: K(\mathcal{\Tilde{A}}) \rightarrow K(\mathcal{A}) \) such that
		\[
		\overline{\varphi}\left( [(A,0,0)]\right) = [A].
		\]
		To find the inverse of \( \overline{\varphi} \), define \( \psi: \operatorname{Obj}(\mathcal{A}) \rightarrow K(\Tilde{\mathcal{A}}) \) such that \(\psi(Z) = [(Z,0,0)]\). Let
		\begin{equation*}
			0 \rightarrow Z' \rightarrow Z \rightarrow Z'' \rightarrow 0
		\end{equation*}
		be a short exact sequence in \( \mathcal{A} \). By the right exactness of \( F \), we get
		\begin{equation*}
			F(Z') \rightarrow F(Z) \rightarrow F(Z'') \rightarrow 0
		\end{equation*}
		an exact sequence in \( \mathcal{C} \). This sequence induces the following diagram:
		$$
		\begin{tikzcd}
			& F(Z') \arrow[r] \arrow[d, "0"'] & F(Z) \arrow[r] \arrow[d, "0"] & F(Z'') \arrow[r] \arrow[d, "0"] & 0 \\ 0 \arrow[r] & 0 \arrow[r]        & 0 \arrow[r] & 0  &  
		\end{tikzcd}
		$$
		which corresponds to the exact short sequence in \( \Tilde{\mathcal{A}} \) given by
		$$0 \rightarrow (Z',0,0) \rightarrow (Z,0,0) \rightarrow (Z'',0,0) \rightarrow 0.$$
		Consequently,
		\begin{equation*}
			[(Z,0,0)] = [(Z',0,0)] + [(Z'',0,0)]
		\end{equation*}
		in \( K(\mathcal{\Tilde{A}}) \). Therefore,
		\begin{equation*}
			\psi(Z) = [(Z,0,0)] = [(Z',0,0)] + [(Z'',0,0)] = \psi(Z') + \psi(Z''),
		\end{equation*}
		meaning that \( \psi \) is additive in short exact sequences of \( \mathcal{A} \). By the universal property of the Grothendieck group, \( \psi \) induces a group homomorphism
		\( \overline{\psi}: K(\mathcal{A}) \rightarrow K(\Tilde{\mathcal{A}}) \) such that
		$$\overline{\psi}([Z]) = [(Z,0,0)].$$
		Thus,
		$$(\overline{\psi} \circ \overline{\varphi})\left( [(A,0,0)] \right) = \overline{\psi}\left( [A] \right) = [(A,0,0)].$$
		$$(\overline{\varphi} \circ \overline{\psi})\left( [Z]\right) = \overline{\varphi}\left( [(Z,0,0)]\right) =  [Z].$$
		This shows that \( \overline{\psi} \) is the inverse of \( \overline{\varphi} \), i.e., \( K(\mathcal{\Tilde{A}}) \cong K(\mathcal{A}) \). Similarly, we prove that \( K(\mathcal{\Tilde{B}}) \cong K(\mathcal{B}) \).
	\end{proof}
	
	\begin{theorem}
		Let \( \mathcal{A} \), \( \mathcal{B} \), and \( \mathcal{C} \) be abelian categories, \( F: \mathcal{A} \rightarrow \mathcal{C} \) a right exact functor, and \( G: \mathcal{B} \rightarrow \mathcal{C} \) a left exact functor. Then, \( K((F/G)) \cong K(\mathcal{A}) \oplus K(\mathcal{B}) \).
		\label{teo_GG_comma}
	\end{theorem}
	
	\begin{proof}
		Follows from \cref{prop_GG_subcat_plen} and \cref{prop_iso_GG-til}.
	\end{proof}
	Hence, we conclude that the Grothendieck group of a comma category splits as the direct sum of the Grothendieck groups of \(\mathcal{A}\) and \(\mathcal{B}\).

	\section{Stability on Comma Categories}
	\label{sect_estab_comma}
	Here we present our main objective: to introduce a notion of stability on comma categories. To this end, we apply concepts from stability theory in abelian categories, making essential use of \cref{prop_exist_HN_func_estab}. The central idea of this section is to determine whether a stability condition in a comma category \((F/G)\) induces stability conditions on \(\mathcal{A}\) and \(\mathcal{B}\), and vice versa. In this section, assume the conditions of \cref{teo_comma_abeliana} hold.
	
	\begin{proposition}
		\label{prop_fun_est-cat_oig_induz_func_estab_comma}
		
		Let \(Z_{\mathcal{A}}: K(\mathcal{A}) \rightarrow \mathbb{C}\) and \(Z_{\mathcal{B}}: K(\mathcal{B}) \rightarrow \mathbb{C}\) be stability functions on \(\mathcal{A}\) and \(\mathcal{B}\), respectively. Then the function \(Z: K((F/G)) \rightarrow \mathbb{C}\) such that
		\[
		Z([A,B, \alpha]) := x Z_{\mathcal{A}}([A]) + y Z_{\mathcal{B}}([B]),
		\]
		where \(x, y \in \mathbb{R}_{>0} \), is a stability function in \((F/G)\). 
	\end{proposition}
	
	\begin{proof}
		The map \(Z\) is clearly an additive homomorphism.
		
		Let \([A] \in K(\mathcal{A})\) and \([B] \in K(\mathcal{B})\). Since \(Z_{\mathcal{A}}\) and \(Z_{\mathcal{B}}\) are stability functions, we have \(\operatorname{Im}Z_{\mathcal{A}}([A]) \geq 0\) and \(\operatorname{Im}Z_{\mathcal{B}}([B]) \geq 0\), hence
		\[
		\operatorname{Im}Z([A,B,\alpha]) = x \operatorname{Im}Z_{\mathcal{A}}([A]) + y \operatorname{Im}Z_{\mathcal{B}}([B]) \geq 0.
		\]
		If \(\operatorname{Im}Z([A,B,\alpha]) = 0\), since \(x, y > 0\), then necessarily \(\operatorname{Im}Z_{\mathcal{A}}([A]) = \operatorname{Im}Z_{\mathcal{B}}([B]) = 0\). This implies \(\operatorname{Re}Z_{\mathcal{A}}([A]) < 0\) and \(\operatorname{Re}Z_{\mathcal{B}}([B]) < 0\), so
		\[
		\operatorname{Re}Z([A,B,\alpha]) = x \operatorname{Re}Z_{\mathcal{A}}([A]) + y \operatorname{Re}Z_{\mathcal{B}}([B]) < 0.
		\]
		Therefore, \(Z\) satisfies the conditions of a stability function on \((F/G)\).
	\end{proof}
	
	Under the assumptions of the proposition above, we can also define the slope in a comma category:
	\[\mu([A,B,\alpha]) = - \dfrac{\operatorname{Re}Z([A,B,\alpha])}{\operatorname{Im}Z([A,B,\alpha])} = - \dfrac{x\operatorname{Re}Z_{\mathcal{A}}([A]) + y\operatorname{Re}Z_{\mathcal{B}}([B])}{x\operatorname{Im}Z_{\mathcal{A}}([A]) + y\operatorname{Im}Z_{\mathcal{B}}([B])},\]
	when \(\operatorname{Im}Z([A,B,\alpha]) \neq 0\) and infinity otherwise.
	
	\begin{lemma}
		\label{lem_FE_sist_cat_inic_FE}
		Let \(Z: K((F/G)) \rightarrow \mathbb{C}\) be a stability function in \((F/G)\). Then \(Z_{\mathcal{A}}:K(\mathcal{A}) \rightarrow \mathbb{C}\) such that \(Z_{\mathcal{A}}([A]) = Z([(A,0,0)])\) is a stability function in \(\mathcal{A}\) and \(Z_{\mathcal{B}}:K(\mathcal{B}) \rightarrow \mathbb{C}\) such that \(Z_{\mathcal{B}}([B]) = Z([(0,B,0)])\) is a stability function in \(\mathcal{B}\).
	\end{lemma}
	\begin{proof}
		We will only show that \(Z_{\mathcal{A}}\) is a stability function; for \(Z_{\mathcal{B}}\), it is analogous. For all \(A \in \mathcal{A}\), we have \(\operatorname{Im}Z_{\mathcal{A}}([A]) = \operatorname{Im}Z([(A,0,0)]) \geq 0\). If \(\operatorname{Im}Z_{\mathcal{A}}([A]) = 0\) it follows that \(\operatorname{Im}Z([(A,0,0)]) = 0\), then \(\operatorname{Re}Z([(A,0,0)]) < 0\), i.e., \(\operatorname{Re}Z_{\mathcal{A}}([A]) < 0\). Therefore, \(Z_{\mathcal{A}}([A])\) is a stability function in \(\mathcal{A}\).
	\end{proof}

	\begin{proposition}
		\label{prop_func_estb_comma_induz_fun_est_categ_origi}
		Every stability function \(Z\) on \((F/G)\) induces stability functions \(Z_{\mathcal{A}}\) and \(Z_{\mathcal{B}}\) on \(\mathcal{A}\) and \(\mathcal{B}\), respectively, given by
		\[
		Z_{\mathcal{A}}([A]) := Z([(A,0,0)]) \quad \mbox{ and } \quad Z_{\mathcal{B}}([B]) := Z([(0,B,0)]).
		\]
	\end{proposition}
	\begin{proof}
		Let \(Z: K((F/G)) \rightarrow \mathbb{C}\) be a stability function. By \cref{GG_alf} and using the fact that \(Z\) is a group homomorphism, we have
		\begin{eqnarray*}
			Z([(A,B,\alpha)]) = Z([(A,B,0)]) 
			& = &
			Z([(A,0,0)] + [(0,B,0)]) \\
			& = &
			Z([(A,0,0)]) + Z([(0,B,0)]).
		\end{eqnarray*}
		Define \(Z_{\mathcal{A}}([A]) := Z([(A,0,0)])\) and \(Z_{\mathcal{B}}([B]) := Z([(0,B,0)])\). Then by \cref{lem_FE_sist_cat_inic_FE}, these are stability functions on \(\mathcal{A}\) and \(\mathcal{B}\), respectively.
	\end{proof}

	\begin{proposition}
		\label{prop_A_B_noeth_sist_noeth}
		If \(\mathcal{A}\) and \(\mathcal{B}\) are noetherian categories, then \((F/G)\) is noetherian.
	\end{proposition}
	\begin{proof}
		Suppose
		\[ (A_{0}, B_{0},\alpha_{0}) \subseteq (A_{1}, B_{1},\alpha_{1}) \subseteq \cdots \subseteq (A_{i}, B_{i},\alpha_{i}) \subseteq \cdots \ \mbox{ in } \ (F/G). \]
		By \cref{mono_epi_comma}, we have
		\[A_{0} \subseteq A_{1} \subseteq \cdots \subseteq A_{i} \subseteq \cdots  \ \mbox{ in } \ \mathcal{A}.\]
		\[B_{0} \subseteq B_{1} \subseteq \cdots \subseteq B_{i} \subseteq \cdots \ \mbox{ in } \ \mathcal{B}.\]
		Since \(\mathcal{A}\) is noetherian, there exists \(m \in \mathbb{N}\) such that the monomorphisms \(A_{i} \subseteq A_{i+1}\) are isomorphisms for all  \(i \geq m\). Since \(\mathcal{B}\) is noetherian, there exists \(n \in \mathbb{N}\) such that the monomorphisms \(B_{j} \subseteq B_{j+1}\) are isomorphisms for all  \(j \geq n\). Let \(r = \operatorname{max}\{m, n\}\). Thus, the monomorphisms \(A_{k} \subseteq A_{k+1}\) and \(B_{k} \subseteq B_{k+1}\) are isomorphisms for all  \(k \geq r\). By \cref{cor_iso_A_B_imp_iso_comma} we have that the monomorphisms \((A_{k}, B_{k},\alpha_{k}) \subseteq (A_{k+1}, B_{k+1},\alpha_{k+1})\) are isomorphisms for all \(k \geq r\).  Therefore, \((F/G)\) is noetherian.
	\end{proof}
	
	\begin{theorem}
		\label{teo_HN_comma}
		Let \(\mathcal{A}\) and \(\mathcal{B}\) be noetherian categories with \(Z_{\mathcal{A}}\) a stability function in \(\mathcal{A}\) and \(Z_{\mathcal{B}}\) a stability function in \(\mathcal{B}\). If the image of \(\operatorname{Im}Z_{\mathcal{A}} + \operatorname{Im}Z_{\mathcal{B}}\) is discrete in \(\mathbb{R}\), then every object in \((F/G)\) admits an \emph{HN} filtration with respect to the stability function \(Z = Z_{\mathcal{A}} + Z_{\mathcal{B}}\).
	\end{theorem}
	\begin{proof}
		By \cref{prop_A_B_noeth_sist_noeth}, we know that \((F/G)\) is noetherian. Moreover, we have  \(\operatorname{Im}Z = \operatorname{Im}Z_{\mathcal{A}} + \operatorname{Im}Z_{\mathcal{B}}\), where the image of \(\operatorname{Im}Z_{\mathcal{A}} + \operatorname{Im}Z_{\mathcal{B}}\) is discrete in \(\mathbb{R}\), so the image of \(\operatorname{Im}Z\) is discrete in \(\mathbb{R}\). Therefore, by \cref{prop_exist_HN_func_estab}, we conclude that every object in \((F/G)\) admits an HN filtration with respect to \(Z\).
	\end{proof}
	
	\begin{remark}
		\label{obs_conj_int_disc_param_racion}
		It can be shown that the set
		\[
		\{(n,m) \in \mathbb{Z}^{2}\,|\, \alpha n +\beta m \mbox{ with }  \alpha,\beta \in \mathbb{R}\}
		\]
		is discrete if, and only if, \(\alpha, \beta \in \mathbb{Q}\).
		As such, if \(\mathcal{A}\) and \(\mathcal{B}\) are noetherian categories and the images \(\operatorname{Im} Z_{\mathcal{A}}(\mathcal{A}), \operatorname{Im} Z_{\mathcal{B}}(\mathcal{B}) \subset \mathbb{Z}\), then the image of \(\alpha \operatorname{Im} Z_{\mathcal{A}} + \beta \operatorname{Im} Z_{\mathcal{B}}\) is discrete for \(\alpha, \beta \in \mathbb{Q}\).
		In particular, by \cref{teo_HN_comma}, every object in \((F/G)\)
		admits an HN filtration with respect to the stability function \(Z = \alpha Z_{\mathcal{A}} + \beta Z_{\mathcal{B}}\) when \(\alpha, \beta \in \mathbb{Q}\).
	\end{remark}
	
	\begin{example}
		Let \(X\) be a smooth projective curve and consider the category of coherent systems (see \cref{exe_sist_coer}). Then \(\mathcal{A} = \operatorname{Vect}_{\mathbb{C}}^{f} = \mathcal{C}\) and \(\mathcal{B} = \operatorname{Coh(X)}\). 
		Let \(Z_{\mathcal{A}}([\Gamma]) = - \alpha\operatorname{dim}(\Gamma)\), with \(\alpha > 0\), and \(Z_{\mathcal{B}}([\mathscr{F}]) = -\operatorname{deg}(\mathscr{F}) + i\operatorname{rk}(\mathscr{F}) \) be stability functions in \(\mathcal{A}\) and \(\mathcal{B}\), respectively. In this case, by \cref{prop_fun_est-cat_oig_induz_func_estab_comma} we have \(Z([(\Gamma, \sigma, \mathscr{F})]) = Z_{\mathcal{A}}(\Gamma) + Z_{\mathcal{B}}(\mathscr{F}) = -(\operatorname{deg}(\mathscr{F})+\alpha \operatorname{dim}(\Gamma)) + i \operatorname{rk}(\mathscr{F})\) a stability function on coherent systems. Moreover, \(\mathcal{A}\) and \(\mathcal{B}\) are noetherian categories and \(\operatorname{Im}Z_{\mathcal{A}}(\Gamma) + \operatorname{Im}Z_{\mathcal{B}}(\mathscr{F}) = \operatorname{rk}(\mathscr{F})\). Since \(\operatorname{rk}(\mathscr{F}) \subset \mathbb{Z}\)  we have \(\operatorname{Im}Z_{\mathcal{A}} + \operatorname{Im}Z_{\mathcal{B}}\) discrete in \(\mathbb{R}\). Therefore, by \cref{teo_HN_comma}, the coherent systems admit an \emph{HN} filtration with respect to the stability function \(Z([(\Gamma, \sigma, \mathscr{F})]) = -(\operatorname{deg}(\mathscr{F})+\alpha \operatorname{dim}(\Gamma)) + i \operatorname{rk}(\mathscr{F})\).
	\end{example}
	
		%

	\begin{proposition}
		\label{prop_comma_noeth_A_B_noeth}
		If \((F/G)\) is noetherian then \(\mathcal{A}\) and \(\mathcal{B}\) are noetherian categories.
	\end{proposition}
	\begin{proof}
		Suppose 
		\[A_{0} \subseteq A_{1} \subseteq \cdots \subseteq A_{i} \subseteq \cdots \ \mbox{ in } \ \mathcal{A}.\]
		\[B_{0} \subseteq B_{1} \subseteq \cdots \subseteq B_{i}\subseteq \cdots \ \mbox{ in } \ \mathcal{B}.\]
		For all \(j \in \mathbb{N}\), the diagram
		\begin{equation*}
			\begin{tikzcd}
				F(A_{j}) \arrow[r] \arrow[d, "0"'] & F(A_{j+1}) \arrow[d, "0"] \\
				G(B_{j}) \arrow[r]                 & G(B_{j+1})               
			\end{tikzcd}
		\end{equation*}
		is commutative, i.e., for all \(j \in \mathbb{N}\) we have \((A_{j}, B_{j}, 0) \rightarrow (A_{j+1}, B_{j+1}, 0)\) a morphism in \((F/G)\). By \cref{lem_morfi_sist_mono_cat_ini_mono_sist}
		follows that 
		\[ (A_{0}, B_{0},0) \subseteq (A_{1}, B_{1},0) \subseteq \cdots \subseteq (A_{i}, B_{i},0) \subseteq \cdots  \ \mbox{ in } \ (F/G). \]
		Since \((F/G)\) is noetherian, there exists \(m \in \mathbb{N}\) such that the monomorphisms \((A_{i},B_{i},0) \subseteq  (A_{i+1},B_{i+1},0)\) are isomorphisms for all \(i \geq m\). By \cref{cor_iso_comma_imp_iso_A_B} we have that the monomorphisms \(A_{i} \subseteq A_{i+1}\) and \(B_{i} \subseteq B_{i+1}\) are isomorphisms for all \(i \geq m\). Therefore \(\mathcal{A}\) and \(\mathcal{B}\) are noetherian categories.
	\end{proof}
	
	\begin{corollary}
		\label{cor_cond_nec_suf_comma_noether}
		\(\mathcal{A}\) and	\(\mathcal{B}\) are noetherian if, and only if, \((F/G)\) is noetherian.
	\end{corollary}
	\begin{proof}
		Follows from \cref{prop_A_B_noeth_sist_noeth} and \cref{prop_comma_noeth_A_B_noeth}.
	\end{proof}
	
	\begin{theorem}
		\label{teo_HN_sist_cat_inic_HN}
		Let \((F/G)\) be a noetherian category, and let \(Z: K((F/G)) \rightarrow \mathbb{C}\) be a stability function in \((F/G)\) such that the image of the imaginary part of \(Z\) is discrete in \(\mathbb{R}\). Then each object in \(\mathcal{A}\) admits an \emph{HN} filtration with respect to the stability function \(Z_{\mathcal{A}}([A]) = Z([(A,0,0)]) \) and each object in \(\mathcal{B}\) admits an \emph{HN} filtration with respect to the stability function \(Z_{\mathcal{B}}([B]) = Z([(0,B,0)]) \).
	\end{theorem}
	\begin{proof}
		By \cref{lem_FE_sist_cat_inic_FE} we have \(Z_{\mathcal{A}}\) and \(Z_\mathcal{B}\) stability functions in \(\mathcal{A}\) and \(\mathcal{B}\), respectively. By \cref{prop_comma_noeth_A_B_noeth}, we have \(\mathcal{A}\) and \(\mathcal{B}\) noetherian categories. Since the image of the imaginary part of \(Z\) is discrete in \(\mathbb{R}\), it follows that the images of \(\operatorname{Im}Z_{\mathcal{A}}\) and \(\operatorname{Im}Z_{\mathcal{B}}\) are discrete in \(\mathbb{R}\). The result follows from \cref{prop_exist_HN_func_estab}.
	\end{proof}
	
	Let \(X\) be a smooth projective curve. Following \cite{newstead2011}, one can define a stability function on coherent systems on \(X\) with respect to a positive real parameter \(\alpha\). More precisely, the {\boldmath{\(\alpha\)}}\textbf{-slope} of a coherent system \((\Gamma, \sigma, \mathscr{F})\) is given by
	\[
	\mu_{\alpha}([(\Gamma, \sigma, \mathscr{F})]) =
	\begin{cases}
		\dfrac{\operatorname{deg}(\mathscr{F})}{\operatorname{rk}(\mathscr{F})} + \alpha \dfrac{\operatorname{dim}(\Gamma)}{\operatorname{rk}(\mathscr{F})}, & \text{if } \operatorname{rk}(\mathscr{F}) \neq 0,\\[1mm]
		+\infty, & \text{if } \operatorname{rk}(\mathscr{F}) = 0.
	\end{cases}
	\]
	Correspondingly, the stability function on coherent systems with respect to \(\alpha > 0\) is defined as
	\[
	Z_{\alpha}([(\Gamma, \sigma, \mathscr{F})]) = -\big(\operatorname{deg}(\mathscr{F}) + \alpha \operatorname{dim}(\Gamma)\big) + i \operatorname{rk}(\mathscr{F}).
	\]
	
	\begin{example}
		Let \(X\) be a smooth projective curve and consider the category of coherent systems. Let \(\mathcal{A} = \operatorname{Vect}_{\mathbb{C}}^{f} = \mathcal{C}\) and \(\mathcal{B} = \operatorname{Coh}(X)\) and \(\alpha > 0\). A stability function on coherent systems is \(Z_{\alpha}([(\Gamma, \sigma, \mathscr{F})]) = -\big(\operatorname{deg}(\mathscr{F}) + \alpha \operatorname{dim}(\Gamma)\big) + i \operatorname{rk}(\mathscr{F})\). Since \(\mathcal{A}\) and \(\mathcal{B}\) are noetherian categories, by \cref{cor_cond_nec_suf_comma_noether} the category of coherent systems is also noetherian. 
		Note that \(\operatorname{Im} Z_{\alpha}([(\Gamma, \sigma, \mathscr{F})]) = \operatorname{rk}(\mathscr{F}) \subset \mathbb{Z}\) is discrete in \(\mathbb{R}\). Hence, by \cref{teo_HN_sist_cat_inic_HN}, each object of \(\mathcal{A}\) admits an \emph{HN} filtration with respect to stability function \(Z_{\mathcal{A}}([\Gamma]) = -\alpha \operatorname{dim}(\Gamma)\), and each object of \(\mathcal{B}\) admits an \emph{HN} filtration with respect to stability function \(Z_{\mathcal{B}}([\mathscr{F}]) = -\operatorname{deg}(\mathscr{F}) + i \operatorname{rk}(\mathscr{F})\).
	\end{example}

	\section{Jordan-Hölder Filtrations on Comma Categories}
	\label{sect_JH_comma}
	We now apply the concept of Jordan–Hölder filtration in the context of comma categories. Specifically, we investigate whether the existence of a Jordan–Hölder filtration for each object in a comma category implies the existence of such filtrations in the original categories, and vice versa. 
	
	\begin{proposition}
		\label{prop_A_B_art_comma_artin}
		If \(\mathcal{A}\) and \(\mathcal{B}\) are artinian, then \((F/G)\) is artinian.
	\end{proposition}
	\begin{proof}
		Suppose
		\[ \cdots \subseteq (A_{i}, B_{i},\alpha_{i}) \subseteq (A_{i-1}, B_{i-1},\alpha_{i-1}) \subseteq \cdots \subseteq (A_{1}, B_{1},\alpha_{1}) \subseteq (A_{0}, B_{0},\alpha_{0})  \]
		in \( (F/G)\). By \cref{mono_epi_comma}, we have
		\[\cdots \subseteq A_{i} \subseteq A_{i-1} \subseteq \cdots \subseteq A_{1}\subseteq A_{0} \ \mbox{ in } \ \mathcal{A}.\]
		\[\cdots \subseteq B_{i} \subseteq B_{i-1} \subseteq \cdots \subseteq B_{1} \subseteq B_{0} \ \mbox{ in } \ \mathcal{B}.\]
		Since \(\mathcal{A}\) is artinian, there exists \(m \in \mathbb{N}\) such that the monomorphisms \(A_{i+1} \subseteq A_{i}\) are isomorphisms for all  \(i \geq m\). Since \(\mathcal{B}\) is artinian, there exists \(n \in \mathbb{N}\) such that the monomorphisms \(B_{j+1} \subseteq B_{j}\) are isomorphisms for all  \(j \geq n\). Let \(r = \operatorname{max}\{m, n\}\). Thus, the monomorphisms \(A_{k+1} \subseteq A_{k}\) and \(B_{k+1} \subseteq B_{k}\) are isomorphisms for all  \(k \geq r\). By \cref{cor_iso_A_B_imp_iso_comma} we have that the monomorphisms \((A_{k+1}, B_{k+1},\alpha_{k+1}) \subseteq (A_{k}, B_{k},\alpha_{k})\) are isomorphisms for all \(k \geq r\). Therefore, \((F/G)\) is artinian.
	\end{proof}

	\begin{theorem}
		\label{teo_JH_A_B_impl_JH_comma}
		If \( \mathcal{A} \) and \( \mathcal{B} \) are noetherian and artinian categories, then each object in \( (F/G) \) admits a Jordan-Hölder filtration.
	\end{theorem}
	
	\begin{proof}
		By \cref{prop_A_B_noeth_sist_noeth} and \cref{prop_A_B_art_comma_artin}, we know that \( (F/G) \) is both a noetherian and an artinian category. Moreover, by \cref{lem_exts_JH}, every object in \( (F/G) \) admits a Jordan-Hölder filtration.
	\end{proof}
	
	\begin{proposition}
		\label{prop_comma_artin_A_B_artin}
		If \((F/G)\) is an artinian category then \(\mathcal{A}\) and \(\mathcal{B}\) are artinian categories.
	\end{proposition}
	\begin{proof}
		Suppose 
		\[\cdots \subseteq A_{i} \subseteq A_{i-1} \subseteq \cdots \subseteq A_{1}\subseteq A_{0} \ \mbox{ in } \ \mathcal{A}.\]
		\[\cdots \subseteq B_{i} \subseteq B_{i-1} \subseteq \cdots \subseteq B_{1}\subseteq B_{0} \ \mbox{ in } \ \mathcal{B}.\]
		For all \(j \in \mathbb{N}\), the diagram
		\begin{equation*}
			\begin{tikzcd}[sep=large]
				F(A_{j+1}) \arrow[r] \arrow[d, "0"'] & F(A_{j}) \arrow[d, "0"] \\
				G(B_{j+1}) \arrow[r]                 & G(B_{j}) \end{tikzcd}
		\end{equation*}
		is commutative, i.e., for all \(j \in \mathbb{N}\) we have \((A_{j+1}, B_{j+1}, 0) \rightarrow (A_{j}, B_{j}, 0)\) a morphism in \((F/G)\). By \cref{lem_morfi_sist_mono_cat_ini_mono_sist}
		follows that 
		\[ \cdots \subseteq (A_{i}, B_{i},0) \subseteq (A_{i-1}, B_{i-1},0) \subseteq \cdots \subseteq (A_{1}, B_{1},0) \subseteq (A_{0}, B_{0},0)  \ \mbox{ in } \ (F/G). \]
		Since \((F/G)\) is artinian, there exists \(m \in \mathbb{N}\) such that the monomorphisms 
		\[(A_{i+1},B_{i+1},0) \subseteq (A_{i},B_{i},0)\]
		 are isomorphisms in \((F/G)\) for all \(i \geq m\). By \cref{cor_iso_comma_imp_iso_A_B} we have that the monomorphisms \(A_{i+1} \subseteq A_{i}\) are isomorphisms in \(\mathcal{A}\) and the monomorphisms \(B_{i+1} \subseteq B_{i}\) are isomorphisms in \(\mathcal{B}\) for all \(i \geq m\). Therefore, \(\mathcal{A}\) and \(\mathcal{B}\) are artinian categories.
	\end{proof}
	
	\begin{corollary}
		\label{cor_cond_nec_suf_comma_artin}
		\(\mathcal{A}\) and	\(\mathcal{B}\) are artinian if, and only if, \((F/G)\) is artinian.
	\end{corollary}
	\begin{proof}
		Follows from \cref{prop_A_B_art_comma_artin} and \cref{prop_comma_artin_A_B_artin}.
	\end{proof}
	
	\begin{theorem}
		\label{teo_JH_comma_impl_JH_A_B}
		If \((F/G)\) is noetherian and artinian, then each object of \(\mathcal{A}\) and \(\mathcal{B}\) admits a Jordan-Hölder filtration.
	\end{theorem}
	\begin{proof}
		By \cref{prop_comma_noeth_A_B_noeth} and \cref{prop_comma_artin_A_B_artin}, we know that \( \mathcal{A} \) and \( \mathcal{B} \) are noetherian and artinian categories, and by \cref{lem_exts_JH}, we have that each object in \( \mathcal{A} \) and \( \mathcal{B} \) admits a Jordan-Hölder filtration.
	\end{proof}
	
	\begin{corollary}
		\label{cor_cond_nec_suf_comma_JH}
		Each object of \(\mathcal{A}\) and \(\mathcal{B}\) admits a Jordan-Hölder filtration if and only if each object of \((F/G)\) admits a Jordan-Hölder filtration.
	\end{corollary}
	\begin{proof}
		By \cref{lem_exts_JH}, each object of \(\mathcal{A}\) and \(\mathcal{B}\) admits a Jordan-Hölder filtration if and only if \(\mathcal{A}\) and \(\mathcal{B}\) are noetherian and artinian. By \cref{cor_cond_nec_suf_comma_noether} and \cref{cor_cond_nec_suf_comma_artin}, this holds if and only if \((F/G)\) is noetherian and artinian. Finally, by \cref{lem_exts_JH}, this holds if and only if each object of \((F/G)\) admits a Jordan-Hölder filtration.
	\end{proof}

	\section{Variations on the Definition of Comma Categories}
	\label{sect_Var_Def_Comma}
	Inspired by \cite{le1993systemes}, which defines the concept of \textbf{coherent cosystems}, an adaptation of coherent systems with a contravariant functor, we modify the definition of our comma categories to introduce the analogous concept of \textbf{co-comma categories}. 
	Moreover, we also discuss an alternative approach to introducing contravariant functors into the definition, showing how the permutation between covariant and contravariant functors allows us to obtain three different categories instead of just one.
	
	\begin{definition}[coherent cosystems]
		\label{def_cos_coe}
		Let \(X\) be a projective variety of dimension \(n\) and let \(d\) and \(c\) be two non-negative integers such that \(d + c = n\). A \textbf{coherent cosystem} of dimension \( d = n - c \) on \( X \) is given by a pair \( (\Gamma, \mathscr{F}) \) consisting of a coherent algebraic sheaf \(\mathscr{F}\) of dimension \( d \) on \( X \) and a vector subspace \(\Gamma \subset \operatorname{Ext}^c(\mathscr{F}, \mathcal{O}_X)\). A \textbf{morphism of coherent cosystems} \( (\Gamma, \mathscr{F}) \rightarrow (\Gamma', \mathscr{F}') \) of dimension \( d \) is given by a morphism of algebraic sheaves \( g: \mathscr{F} \rightarrow \mathscr{F}' \) inducing a commutative diagram:
		
		\begin{equation*}
			\begin{tikzcd}
				\Gamma' \arrow[d] \arrow[rr]                                                                                      &  & \Gamma \arrow[d]                                      \\
				{\operatorname{Ext}^{c}(\mathscr{F'},\mathcal{O}_{X})} \arrow[rr, "{\operatorname{Ext}^{c}(g,\mathcal{O}_{X})}"'] &  & {\operatorname{Ext}^{c}(\mathscr{F},\mathcal{O}_{X})}
			\end{tikzcd}
		\end{equation*}
	\end{definition}

	\subsection{Co-comma Categories}
	Let \(\mathcal{A}\), \(\mathcal{B}\), and \(\mathcal{C}\) be categories and \(F: \mathcal{A} \rightarrow \mathcal{C}\) be a covariant functor and \(G: \mathcal{B} \rightarrow \mathcal{C}\) be a contravariant functor. We define \((F \backslash G)\) as a category such that
	\begin{itemize}
		\item Objects of \((F \backslash G)\): are triples \((A,B,\alpha)\), where \(A \in \operatorname{Obj}(\mathcal{A})\), \(B \in \operatorname{Obj}(\mathcal{B})\) and \(\alpha: F(A) \rightarrow G(B) \in \operatorname{Mor}(\mathcal{C})\).
		
		\item Morphisms of \((F\backslash G)\): a morphism between \((A,B,\alpha)\) and \((A',B',\alpha')\), objects of \((F\backslash G)\), is a pair \((f,g)\) with \(f: A' \rightarrow A \in \operatorname{Mor}(\mathcal{A})\) and \(g: B \rightarrow B' \in \operatorname{Mor}(\mathcal{B})\) such that the following diagram commutes:
		\begin{equation*}
			\begin{tikzcd}
				F(A') \arrow[r, "F(f)"] \arrow[d, "\alpha'"'] & F(A) \arrow[d, "\alpha"] \\
				G(B') \arrow[r, "G(g)"']                      & G(B)                    
			\end{tikzcd}
		\end{equation*} 
	\end{itemize}
	The proof that \((F \backslash G)\) is a category, which we call the \textbf{co-comma category}, is analogous to the proof that \((F/G)\) is a category. The identity morphism of \((F \backslash G)\) is the same as the identity morphism of \((F/G)\), i.e., \(1_{(A,B,\alpha)}: (A,B,\alpha) \rightarrow (A,B,\alpha) \) is given by \(1_{(A,B,\alpha)} = (1_A, 1_B)\). The composition between \((f,g): (A,B,\alpha) \rightarrow (A',B',\alpha')\) and \((f',g'): (A',B',\alpha') \rightarrow (A'',B'',\alpha'')\), both morphisms in \((F \backslash G)\), is defined by 
	\[
	(f',g') \circ (f,g) := (f \circ f', g' \circ g).
	\]
	
	\subsubsection{An Example of a Co-comma Category}
	Following \cite{huybrechts2010geometry}, we recall the notion of a framed module:
	\begin{definition}
		\label{def_framed_module}
		Let \(X\) be a smooth projective variety over an algebraically closed field \(k\) of characteristic zero. Fix an ample invertible sheaf \(\mathcal{O}_X(1)\) and a non-trivial coherent sheaf \(\mathscr{F}\). A \textbf{framed module} is a pair \((\mathscr{E}, \varphi)\) consisting of a coherent sheaf \(\mathscr{E}\) and a 
		homomorphism \(\varphi : \mathscr{E} \to \mathscr{F}\). 
		
		A morphism of framed modules \((\mathscr{E},\varphi) \to (\mathscr{E}',\varphi')\) 
		is a morphism of coherent sheaves \(f : \mathscr{E} \to \mathscr{E}'\) such that the following diagram commutes:
		\[
		\begin{tikzcd}
			\mathscr{E} \arrow[rr, "f"] \arrow[rd, "\varphi"'] &   & \mathscr{E}' \arrow[ld, "\varphi'"] \\ & \mathscr{F} &       \end{tikzcd}
		\]
		i.e., \(\varphi' \circ f = \varphi\).
	\end{definition}
	
	We can extend this notion by introducing a vector space. Fix a coherent sheaf \(\mathscr{F}\). 
	A \textit{generalized framed module} is a triple \((V, \mathscr{E}, \varphi)\), where \(V\) is a vector space, \(\mathscr{E}\) is a coherent sheaf, and 
	\(\varphi: V \to \operatorname{Hom}(\mathscr{E}, \mathscr{F})\).
	
	A morphism
	\[
	(g,f): (V, \mathscr{E}, \varphi) \longrightarrow (V', \mathscr{E}', \varphi')
	\]
	consists of a linear map \(g: V' \to V\) and a morphism of coherent sheaves \(f: \mathscr{E} \to \mathscr{E}'\) 
	such that the following diagram commutes:
	\[
	\begin{tikzcd}
		V' \arrow[d, "\varphi'"'] \arrow[rr, "g"]                                                         &  & V \arrow[d, "\varphi"]                        \\
		{\operatorname{Hom}(\mathscr{E}',\mathscr{F})} \arrow[rr, "{\operatorname{Hom}(f,\mathscr{F})}"'] &  & {\operatorname{Hom}(\mathscr{E},\mathscr{F})}
	\end{tikzcd}
	\]
	
	Recall the canonical identifications
	\[
	\operatorname{Hom}(V, \operatorname{Hom}(\mathscr{E}, \mathscr{F})) 
	\cong \operatorname{Hom}(\mathscr{E},\mathscr{F}) \otimes V ^{*} 
	\cong \operatorname{Hom}(\mathscr{E}, V^{*} \otimes \mathscr{F}).
	\]
	In the case \(\dim V = 1\), we have \(V^{*} \otimes \mathscr{F} \cong \mathscr{F}\), so
	\[
	\operatorname{Hom}(V, \operatorname{Hom}(\mathscr{E}, \mathscr{F})) \cong \operatorname{Hom}(\mathscr{E},\mathscr{F}),
	\]
	and the previous diagram reduces to
	\[
	\begin{tikzcd}
		\mathscr{E} \arrow[r, "f"] \arrow[d, "\varphi"'] & \mathscr{E}' \arrow[d, "\varphi'"] \\
		\mathscr{F} \arrow[r, "\operatorname{Id}"']      & \mathscr{F}                       
	\end{tikzcd}
	\]
	i.e., \(\varphi' \circ f = \varphi\). Hence, when \(\dim V = 1\), the generalized notion recovers the original framed module.
	
	Finally, this construction naturally fits into the framework of co-comma categories. Indeed, let \(\mathcal{A} = \mathcal{C} = \operatorname{Vect}\) and \(\mathcal{B} = \operatorname{Coh}(X)\), with \(F = \operatorname{Id}\) and \(G = \operatorname{Hom}(-, \mathscr{F})\) a contravariant functor.
	
	\subsection{Other Constructions}
	
	In addition, we can also define two other categories.
	\begin{example}
		Let \(\mathcal{A}\), \(\mathcal{B}\), and \(\mathcal{C}\) be categories.
		\begin{enumerate}
			\item Let \(F: \mathcal{A} \rightarrow \mathcal{C}\) be a contravariant functor and \(G: \mathcal{B} \rightarrow \mathcal{C}\) be a covariant functor. We define the category \((F|G)\) such that
			\begin{itemize}
				\item Objects of \((F|G)\): are triples \((A,B,\alpha)\), where \(A \in \operatorname{Obj}(\mathcal{A})\), \(B \in \operatorname{Obj}(\mathcal{B})\) and \(\alpha: F(A) \rightarrow G(B) \in \operatorname{Mor}(\mathcal{C})\).
				
				\item Morphisms of \((F|G)\): a morphism between \((A,B,\alpha)\) and \((A',B',\alpha')\) is a pair \((f,g)\) with \(f: A \rightarrow A' \in \operatorname{Mor}(\mathcal{A})\) and \(g: B' \rightarrow B \in \operatorname{Mor}(\mathcal{B})\) such that the following diagram commutes:
				\begin{equation*}
					\begin{tikzcd}
						F(A') \arrow[r, "F(f)"] \arrow[d, "\alpha'"'] & F(A) \arrow[d, "\alpha"] \\
						G(B') \arrow[r, "G(g)"']                      & G(B)                    
					\end{tikzcd}
				\end{equation*} 
			\end{itemize}
			\item Let $F: \mathcal{A} \rightarrow \mathcal{C}$ and $G: \mathcal{B} \rightarrow \mathcal{C}$ be contravariant functors. We define the category \((F-G)\) such that
			\begin{itemize}
				\item Objects of \((F- G)\): are triples $(A,B,\alpha)$, where $A \in \operatorname{Obj}(\mathcal{A})$, $B \in \operatorname{Obj}(\mathcal{B})$ and $\alpha: F(A) \rightarrow G(B) \in \operatorname{Mor}(\mathcal{C})$.
				
				\item Morphisms of \((F- G)\): a morphism between $(A,B,\alpha)$ and $(A',B',\alpha')$ is a pair $(f,g)$ with $f: A \rightarrow A'$, $g: B \rightarrow B'$ such that the following diagram commutes:
				\begin{equation*}
					\begin{tikzcd}
						F(A') \arrow[r, "F(f)"] \arrow[d, "\alpha'"'] & F(A) \arrow[d, "\alpha"] \\
						G(B') \arrow[r, "G(g)"']                      & G(B)                    \end{tikzcd}
				\end{equation*} 
			\end{itemize}
		\end{enumerate}
	\end{example}
	
	\subsection{Conditions for  Co-comma Categories to be Abelian}
	In this section, we present the conditions on the categories and functors involved in the definition of the co-comma category for it to be an abelian category. The proof follows a similar approach to that of a comma category and, for this reason, we omit some unnecessary details.
	\begin{proposition}
		\label{prop_cosist_adit}
		Let \(\mathcal{A}\), \(\mathcal{B}\) and \(\mathcal{C}\) be additive categories, and let \(F: \mathcal{A} \rightarrow \mathcal{C}\) be an additive covariant functor and \(G: \mathcal{B} \rightarrow \mathcal{C}\) be an additive contravariant functor. Then \((F \backslash G)\) is an additive category.
	\end{proposition}
	\begin{proof}
		The proof that \((F \backslash G)\) is a linear category follows similarly to the proof that \((F/G)\) is a linear category.
		
		The zero object of \((F \backslash G)\) is the same as that of \((F/G)\), given by \((0_{\mathcal{A}}, 0_{\mathcal{B}}, 0)\). The proof is also analogous: using that \(0_{\mathcal{B}}\) is an initial object and \(0_{\mathcal{A}}\) and \(0_{\mathcal{C}}\) are final objects we prove that \((0_{\mathcal{A}}, 0_{\mathcal{B}}, 0)\) is an initial object. Finally, using that \(0_{\mathcal{A}}\) and \(0_{\mathcal{C}}\) are initial objects and \(0_{\mathcal{B}}\) is a final object, it follows that \((0_{\mathcal{A}}, 0_{\mathcal{B}}, 0)\) is a final object.
		
		Finally, we will show the existence of finite coproducts in \((F \backslash G)\).  Let $(A, B, \alpha)$ and $(A', B', \alpha')$ be arbitrary objects in \((F \backslash G)\). Consider \((A \times A',\pi_{1}^{A},\pi_{2}^{A})\) the product of \(A\) and \(A'\) in \(\mathcal{A}\) and \((B \oplus B',i_{1}^{B},i_{2}^{B})\) the coproduct of \(B\) and \(B'\) in \(\mathcal{B}\). Since \(F\) is an additive functor, we have \((F(A \times A'), F(\pi_{1}^{A}), F(\pi_{2}^{A}))\) is the product of \(F(A)\) and \(F(A')\), and since \(G\) is an additive contravariant functor, we have \((G(B \oplus B'), G(i_{1}^{B}), G(i_{2}^{B}))\) is the product of \(G(B)\) and \(G(B')\). Thus, we have the following commutative diagram
		\begin{equation*}
			\begin{tikzcd}
				G(B) & G(B \oplus B') \arrow[l, "G(i_{1}^{B})"'] \arrow[r, "G(i_{2}^{B})"]                                                              & G(B') \\
				& F(A \times A') \arrow[lu, "\alpha \circ F(\pi_{1}^{A})"] \arrow[ru, "\alpha' \circ F(\pi_{2}^{A})"'] \arrow[u, "\beta"', dashed] &      
			\end{tikzcd}
		\end{equation*}
		where \(\beta\) comes from the universal property of the product of \(G(B)\) and \(G(B')\). Thus, \(\beta: F(A \times A') \rightarrow G(B \oplus B')\) is unique such that
		\begin{equation}
			\alpha \circ F(\pi_{1}^{A}) = G(i_{1}^{B}) \circ \beta  \ \mbox{ and } \ \alpha' \circ F(\pi_{2}^{A}) = G(i_{2}^{B}) \circ \beta .
			\label{eq_beta_prod_cosist}
		\end{equation}
		The \cref{eq_beta_prod_cosist} shows that \((\pi_{1}^{A}, i_{1}^{B}) : (A,B, \alpha) \rightarrow (A \times A', B \oplus B',\beta)\) and \((\pi_{2}^{A}, i_{2}^{B}): (A',B', \alpha') \rightarrow (A \times A', B \oplus B',\beta)\) are morphisms in \((F \backslash G)\).
		Let $(X, Y, H)$ be an object of \((F \backslash G)\) and let $(f_{1}, g_{1}): (A, B, \alpha) \rightarrow (X, Y, H)$ and $(f_{2}, g_{2}): (A', B', \alpha') \rightarrow (X, Y, H)$ be arbitrary morphisms in 
		\((F \backslash G)\). We will show that there exists a unique morphism \((\varphi, \psi): (A \times A', B \oplus B', \beta) \rightarrow (X, Y, H)\) in \((F \backslash G)\)  such that the diagram below commutes.
		\begin{equation*}
			\begin{tikzcd}
				&  & {(X,Y,H)}  &  &     \\
				{(A,B,\alpha)} \arrow[rru, "{(f_{1},g_{1})}"] \arrow[rr, "{(\pi_{1}^{A},i_{1}^{B})}"'] &  & {(A \times A', B \oplus B', \beta)} \arrow[u, "{(\varphi,\psi)}"', dashed] &  & {(A',B',\alpha')} \arrow[llu, "{(f_{2},g_{2})}"'] \arrow[ll, "{(\pi_{2}^{A},i_{2}^{B})}"]
			\end{tikzcd}
		\end{equation*}
		Since $(f_{1}, g_{1})$ and $(f_{2}, g_{2})$ are morphisms in \((F \backslash G)\), \(F\) is a covariant functor, \(G\) is a contravariant functor and using the universal property of the product of $A$, $A'$ and of the coproduct of $B$, $B'$, there exist unique morphisms $\varphi: X \rightarrow  A \times A'$ and $\psi: B \oplus B' \rightarrow Y$ such that 
		\begin{equation}
			\begin{cases}
				F(f_{1}) = F(\pi_{1}^{A}) \circ F(\varphi) \\
				F(f_{2}) = F(\pi_{2}^{A}) \circ F(\varphi)
			\end{cases}
			\quad\quad
			\begin{cases}
				G(g_{1}) = G(i_{1}^{B}) \circ G(\psi)  \\ 
				G(g_{2}) = G(i_{2}^{B}) \circ G(\psi) 
			\end{cases}
			\label{eq_Fvarphi_Gpsi_coprod_cosist}
		\end{equation}
		By the universal property of the product $(G(B \oplus B'), G(i_{1}^{B}), G(i_{2}^{B}))$ there exist unique $\rho, \rho': F(X) \rightarrow  G(B \oplus B')$ such that
		\begin{equation}
			\begin{cases}
				\alpha \circ F(f_{1}) = G(i_{1}^{B}) \circ \rho' \\
				\alpha' \circ F(f_{2}) = G(i_{2}^{B}) \circ \rho'
			\end{cases}
			\quad\quad
			\begin{cases}
				G(g_{1}) \circ H = G(i_{1}^{B}) \circ \rho  \\ 
				G(g_{2}) \circ H = G(i_{2}^{B}) \circ \rho 
			\end{cases}
			\label{eq_phi_phi'_cosist}
		\end{equation}
		By \cref{eq_Fvarphi_Gpsi_coprod_cosist} and \cref{eq_beta_prod_cosist} we have
		\begin{equation}
			\label{eq_prov_unic_phi_phi'}
			\begin{cases}
				\alpha \circ F(f_{1}) = G(i_{1}^{B}) \circ (\beta \circ F(\varphi))  \\ 
				\alpha' \circ F(f_{2}) = G(i_{2}^{B}) \circ (\beta \circ F(\varphi))
			\end{cases}
			\quad
			\begin{cases}
				G(g_{1}) \circ H = G(i_{1}^{B}) \circ (G(\psi) \circ H) \\
				G(g_{2}) \circ H = G(i_{2}^{B}) \circ (G(\psi) \circ H) 
			\end{cases}
		\end{equation}
		By the uniqueness of \(\rho\) and \(\rho'\) and by \cref{eq_prov_unic_phi_phi'}, it follows that
		\begin{equation}
			\rho = G(\psi) \circ H  \ \mbox{ and } \ \rho' = \beta \circ F(\varphi).
			\label{eq_nov_carc_phi_phi'_cosist}
		\end{equation}
		Since $(f_{1}, g_{1})$ and $(f_{2}, g_{2})$ are morphisms in \((F \backslash G)\), we have
		\begin{equation}
			\alpha \circ F(f_{1}) = G(g_{1}) \circ H \ \mbox{ and } \ \alpha' \circ F(f_{2}) = G(g_{2}) \circ H.
			\label{eq_f12_g12_mors_cosist}
		\end{equation}
		From the \cref{eq_f12_g12_mors_cosist}, \cref{eq_prov_unic_phi_phi'} and \cref{eq_nov_carc_phi_phi'_cosist}, we obtain 
		\[
		\alpha \circ F(f_{1}) = G(i_{1}^{B}) \circ \rho \ \mbox{ and } \ \alpha' \circ F(f_{2}) = G(i_{2}^{B}) \circ \rho .
		\]
		By the uniqueness of $\rho'$ we have $\rho = \rho'$. Thus, from the equation \cref{eq_nov_carc_phi_phi'_cosist}, it follows that $\beta \circ F(\varphi) = G(\psi) \circ H$, that is, $(\varphi, \psi)$ is a morphism in \((F \backslash G)\). Moreover, by construction, $(\varphi, \psi)$ is unique such that
		$$(\varphi,\psi) \circ (\pi_{1}^{A},i_{1}^{B}) = (f_{1}, g_{1}) \ \mbox{ and } \ (\varphi,\psi) \circ (\pi_{2}^{A},i_{2}^{B}) = (f_{2}, g_{2}).$$
		Thus, \(((A \times A', B \oplus B', \beta),(\pi_{1}^{A},i_{1}^{B}),(\pi_{2}^{A},i_{2}^{B}))\) is the coproduct of \((A,B,\alpha)\) and \((A',B',\alpha')\). This way, \((F \backslash G)\) has finite coproducts. Therefore \((F \backslash G)\) is an additive category.
	\end{proof}
	
	\begin{proposition}
		Let $\mathcal{A}$, $\mathcal{B}$ and $\mathcal{C}$ be abelian categories, $F: \mathcal{A} \rightarrow \mathcal{C}$ a right exact functor and $G: \mathcal{B} \rightarrow \mathcal{C}$ an additive contravariant functor. Then, \((F \backslash G)\) admits kernels.
		\label{prop_cosist_kern}
	\end{proposition}
	\begin{proof}
		Let $(f, g): (A, B, \alpha) \rightarrow (A', B', \alpha')$ be an arbitrary morphism in \((F \backslash G)\). Since $f: A' \rightarrow A \in \operatorname{Mor}(\mathcal{A})$ and $g: B \rightarrow B' \in \operatorname{Mor}(\mathcal{B})$, there exist $(\operatorname{Coker} f, \operatorname{coker} f)$ and $(\operatorname{Ker} g, \operatorname{ker} g)$. From the fact that $(f, g)$ is a morphism in \((F \backslash G)\), we have
		\begin{equation}
			\alpha \circ F(f) = G(g) \circ \alpha' .
			\label{eq_prop_kern_alpha_morf}
		\end{equation}
		By \cref{eq_prop_kern_alpha_morf} we have
		\begin{equation*}
			G(\operatorname{ker}g) \circ \alpha \circ  F(f) =  G(\operatorname{ker}g) \circ G(g) \circ \alpha' = G(g \circ \operatorname{ker}g) \circ \alpha' = G(0) \circ\alpha' = 0.
		\end{equation*}
		From the right exactness of \(F\) and the universal property of the cokernel of \(F(f)\), there exists a unique morphism \(\beta: F(\operatorname{Coker} f) \rightarrow G(\operatorname{Ker} g)\) such that the diagram below commutes.
		\begin{equation*}
			\begin{tikzcd}
				F(A') \arrow[r, "F(f)"] & F(A) \arrow[r, "F(\operatorname{coker}f)"] \arrow[d, "G(\operatorname{ker}g) \circ \alpha"'] & F(\operatorname{Coker}f) \arrow[ld, "\beta", dashed] \\
				& G(\operatorname{Ker}g)       &   
			\end{tikzcd}
		\end{equation*}
		i.e.,
		\begin{equation}
			\beta \circ F(\operatorname{coker}f) = G(\operatorname{ker}g) \circ \alpha .
			\label{eq_mor_beta_prop_univ_kern}
		\end{equation}
		We will show that $((\operatorname{Coker} f, \operatorname{Ker} g, \beta), (\operatorname{coker} f, \operatorname{ker} g))$ is the kernel of the morphism $(f, g)$. By \cref{eq_mor_beta_prop_univ_kern}, it follows that $(\operatorname{coker} f, \operatorname{ker} g): (\operatorname{Coker} f, \operatorname{Ker} g, \beta) \rightarrow (A,B,\alpha)$ is a morphism in \((F \backslash G)\). Moreover,
		$$(f,g) \circ (\operatorname{coker}f, \operatorname{ker}g) = (\operatorname{coker}f \circ f, g \circ \operatorname{ker}g) = (0,0) = 0.$$
		Thus, it remains to show the universal property of the kernel. For this, consider $(h, k): (X, Y, H) \rightarrow (A, B, \alpha)$ be a morphism in  \((F \backslash G)\) such that $(f, g) \circ (h, k) = 0$. Since $(h, k)$ is a morphism in  \((F \backslash G)\), we have
		\begin{equation}
			H \circ F(h) = G(k) \circ \alpha.
			\label{eq_morf_hk_cosist}
		\end{equation}
		Since $h \circ f = 0$ and $g \circ k = 0$, by the universal property of the cokernel of $f$ and of the kernel of $g$, there exist unique morphisms $\varphi: \operatorname{Coker} f \rightarrow X$ and $\psi: Y \rightarrow \operatorname{Ker} g$ such that
		\begin{equation}
			h = \varphi \circ \operatorname{coker}f \ \mbox{ and } \ k = \operatorname{ker}g \circ \psi .
			\label{eq_morfi_h_k_prop_univ_kern}
		\end{equation}
		By equation \cref{eq_morfi_h_k_prop_univ_kern}
		\begin{equation}
			F(h) = F(\varphi) \circ F(\operatorname{coker}f)   \ \mbox{ and } \ G(k) = G(\psi) \circ G(\operatorname{ker}g).
			\label{eq_Fh_Gk_morf}
		\end{equation}
		This way by \cref{eq_mor_beta_prop_univ_kern}
		\begin{equation*}
			G(\psi) \circ (G(\operatorname{ker}g) \circ  \alpha) \circ F(f) = G(\psi) \circ (\beta \circ F(\operatorname{coker}f)) \circ F(f) = 0 .
		\end{equation*}
		From the right exactness of \(F\), the universal property of the cokernel of \(F(f)\) and by \cref{eq_Fh_Gk_morf}, there exists a unique \(\Phi: F(\operatorname{Coker}f) \rightarrow G(Y)\) such that
		\begin{equation}
			G(k) \circ \alpha = \Phi \circ F(\operatorname{coker}f).
			\label{eq_Phi_prop_univ_coker_F}
		\end{equation}
		By \cref{eq_mor_beta_prop_univ_kern}, \cref{eq_Fh_Gk_morf}  and \cref{eq_Phi_prop_univ_coker_F}:
		\begin{equation*}
			(G(\psi) \circ \beta ) \circ F(\operatorname{coker}f) = G(\psi) \circ (G(\operatorname{ker}g) \circ \alpha) = G(k) \circ \alpha = \Phi \circ F(\operatorname{coker}f) .
		\end{equation*}
		From the right exactness of \(F\), it follows that
		\begin{equation}
			\Phi = G(\psi) \circ \beta .
			\label{eq_Phi_G_beta}
		\end{equation}
		On the other hand, by \cref{eq_Fh_Gk_morf}, \cref{eq_morf_hk_cosist} and \cref{eq_Phi_prop_univ_coker_F}:
		\begin{equation*}
			H \circ F(\varphi) \circ F(\operatorname{coker}f) = H \circ F(h) = G(k) \circ \alpha = \Phi \circ F(\operatorname{coker}f) .
		\end{equation*}
		Again, using the right exactness of \(F\)
		\begin{equation}
			\Phi = H \circ F(\varphi).
			\label{eq_Phi_F_H}
		\end{equation}
		From equations \cref{eq_Phi_G_beta} and \cref{eq_Phi_F_H}:
		\begin{equation*}
			H \circ F(\varphi) = G(\psi) \circ \beta,
		\end{equation*}
		i.e., \((\varphi,\psi): (X, Y, H) \rightarrow (\operatorname{Coker} f, \operatorname{Ker} g, \beta)\) is a morphism in \((F \backslash G)\). Moreover, by \cref{eq_morfi_h_k_prop_univ_kern}
		\begin{equation}
			(\operatorname{coker}f, \operatorname{ker}g) \circ (\varphi, \psi) = (\varphi\circ \operatorname{coker}f, \operatorname{ker}g \circ \psi) = (h,k).
			\label{eq_psi_phi_sat_prop_ker}
		\end{equation}
		Finally, by construction, $(\varphi, \psi)$ is unique with the property of equation \cref{eq_psi_phi_sat_prop_ker}. Therefore, \((F \backslash G)\) admits kernels.
	\end{proof}
	
	\begin{proposition}
		Let $\mathcal{A}$, $\mathcal{B}$ and $\mathcal{C}$ be abelian categories, $F: \mathcal{A} \rightarrow \mathcal{C}$ an additive functor and $G: \mathcal{B} \rightarrow \mathcal{C}$ a right exact contravariant functor. Then, \((F \backslash G)\) admits cokernels.
		\label{prop_cosist_coker}
	\end{proposition}
	\begin{proof}
		The proof is analogous to the proof of the \cref{prop_cosist_kern}. For any morphism $(f,g): (A,B,\alpha) \rightarrow (A',B',\alpha')$ in \((F \backslash G)\), the cokernel of \((f,g)\) is given by 
		\[
		(\operatorname{Coker}(f,g),\operatorname{coker}(f,g)) =	((\operatorname{Ker} f, \operatorname{Coker} g, \gamma), (\operatorname{ker} f, \operatorname{coker} g)),
		\]
		where $\gamma: F(\operatorname{Ker} f) \rightarrow G(\operatorname{Coker} g)$ is the unique morphism making the diagram below commute.
		\begin{equation*}
			\begin{tikzcd}
				G(\operatorname{Coker}g) \arrow[r, "G(\operatorname{coker}g)"] & G(B') \arrow[r, "G(g)"]  & G(B) \\  & F(\operatorname{Ker}f) \arrow[u, "\alpha' \circ F(\operatorname{ker}f)"'] \arrow[lu, "\gamma", dashed] &   \end{tikzcd}
		\end{equation*}
	\end{proof}
	
	\begin{theorem}
		\label{teo_cosist_abeli}
		Let $\mathcal{A}$, $\mathcal{B}$ and $\mathcal{C}$ be abelian categories, $F: \mathcal{A} \rightarrow \mathcal{C}$ a right exact functor and $G: \mathcal{B} \rightarrow \mathcal{C}$ a right exact contravariant functor. Then, \((F \backslash G)\) is an abelian category.
	\end{theorem}
	\begin{proof}
		By \cref{prop_cosist_adit}, \cref{prop_cosist_kern} and \cref{prop_cosist_coker}, it suffices to show that, for every morphism in \((F \backslash G)\), its induced morphism is an isomorphism. Let \((f,g): (A,B,\alpha) \rightarrow (A',B',\alpha')\) be any morphism in \((F \backslash G)\). By the basic property of categories with kernels and cokernels, there exist unique morphisms \(\varphi: A' \rightarrow \operatorname{Im}f\), \( \overline{f}: \operatorname{Coim}f \rightarrow \operatorname{Im}f\) in \(\mathcal{A}\) and \(\psi: B \rightarrow \operatorname{Im}g\), \( \overline{g}: \operatorname{Coim}g \rightarrow \operatorname{Im}g\) in \(\mathcal{B}\) such that
		\begin{equation}
			\label{eq_morf_ind_cat_A_co_comma}
			f = \beta \circ \varphi \ \mbox{ and } \ \varphi = \overline{f} \circ \theta.
		\end{equation}
		\begin{equation}
			\label{eq_morf_ind_cat_B_co_comma}
			g = \Tilde{\beta} \circ \psi \ \mbox{ and } \ \psi = \overline{g} \circ \Tilde{\theta}.
		\end{equation}
		where \(\theta = \operatorname{coker}(\operatorname{ker}f)\), \(\beta = \operatorname{ker}(\operatorname{coker}f)\), \(\Tilde{\theta} = \operatorname{coker}(\operatorname{ker}g)\) and \(\Tilde{\beta} = \operatorname{ker}(\operatorname{coker}g)\).
		Note that
		\begin{eqnarray*}
			\operatorname{Im}(f,g) = \operatorname{Ker}(\operatorname{coker}(f,g)) 
			&=& \operatorname{Ker}(\operatorname{ker}f, \operatorname{coker}g) \\
			& = & (\operatorname{Coker}(\operatorname{ker}f),\operatorname{Ker}(\operatorname{coker}g),\Delta) \\
			&=&
			(\operatorname{Coim}f, \operatorname{Im}g, \Delta) ,
		\end{eqnarray*}
		where \(\Delta: F(\operatorname{Coim}f) \rightarrow G(\operatorname{Im}g)\) arises from the universal property of the cokernel of \(F(\operatorname{ker}f)\).
		\begin{eqnarray*}
			\operatorname{Coim}(f,g) = \operatorname{Coker}(\operatorname{ker}(f,g)) 
			&=&
			\operatorname{Coker}(\operatorname{coker}f, \operatorname{ker}g) \\
			& = & (\operatorname{Ker}(\operatorname{coker}f),\operatorname{Coker}(\operatorname{ker}g),\Gamma) \\
			& = & 
			(\operatorname{Im}f, \operatorname{Coim}g, \Gamma),
		\end{eqnarray*}
		where \(\Gamma: F(\operatorname{Im}f) \rightarrow G(\operatorname{Coim}g)\) arises from the universal property of the kernel of \(G(\operatorname{coker}g)\). First we will prove that \((\beta \circ \overline{f}, \psi): (A,B,\alpha) \rightarrow (\operatorname{Coim}f, \operatorname{Im}g, \Delta)\) is a morphism in \((F \backslash G)\). Note that
		\[
		\operatorname{coker}(\operatorname{ker}(f,g)) = \operatorname{coker}(\operatorname{coker}f, \operatorname{ker}g) = (\operatorname{ker}(\operatorname{coker}f), \operatorname{coker}(\operatorname{ker}g)) = (\beta, \Tilde{\theta})
		\]
		\[
		\operatorname{ker}(\operatorname{coker}(f,g)) = \operatorname{ker}(\operatorname{ker}f, \operatorname{coker}g) = (\operatorname{coker}(\operatorname{ker}f), \operatorname{ker}(\operatorname{coker}g)) = (\theta, \Tilde{\beta})
		\]
		are morphisms in \((F \backslash G)\). Then the following diagram commutes:
		\begin{equation*}
			\begin{tikzcd}
				F(A') \arrow[r, "F(\theta)"] \arrow[d, "\alpha'"'] & F(\operatorname{Coim}f) \arrow[d, "\Delta"] &  & F(\operatorname{Im}f) \arrow[r, "F(\beta)"] \arrow[d, "\Gamma"'] & F(A) \arrow[d, "\alpha"] \\
				G(B') \arrow[r, "G(\Tilde{\beta})"']               & G(\operatorname{Im}g)                       &  & G(\operatorname{Coim}g) \arrow[r, "G(\Tilde{\theta})"']          & G(B)                    
			\end{tikzcd}
		\end{equation*}
		From the diagram e using that \((f,g)\) is a morphism in \((F \backslash G)\) we have
		\begin{equation}
			\label{eq_morf_fg_co_comma}
			\alpha \circ F(f) = G(g) \circ \alpha'.
		\end{equation}
		\begin{equation}
			\label{eq_morf_coim_co_comma}
			\Delta \circ F(\theta) = G(\Tilde{\beta}) \circ \alpha'.
		\end{equation}
		\begin{equation}
			\label{eq_morf_im_co_comma}
			\alpha \circ F(\beta) = G(\Tilde{\theta}) \circ \Gamma.
		\end{equation}
		From \cref{eq_morf_ind_cat_A_co_comma}, we have
		\[
		\alpha \circ F(f) = \alpha \circ F(\beta \circ \varphi) = \alpha \circ F(\beta) \circ F(\overline{f} \circ \theta) = \alpha \circ F(\beta \circ \overline{f}) \circ F(\theta).
		\]
		By \cref{eq_morf_ind_cat_B_co_comma} and \cref{eq_morf_coim_co_comma}:
		\[
		G(g) \circ \alpha' = G(\Tilde{\beta} \circ \psi) \circ \alpha' = G(\psi) \circ  G(\Tilde{\beta}) \circ \alpha' = G(\psi) \circ \Delta \circ F(\theta).
		\]
		From \cref{eq_morf_fg_co_comma}, it follows that
		\[
		\alpha \circ F(\beta \circ \overline{f}) \circ F(\theta) = G(\psi) \circ \Delta \circ F(\theta).
		\]
		Since \(\theta\) is an epimorphism and \(F\) is a right exact functor, we have:
		\begin{equation}
			\label{eq_morf_varphi_psi_co_comma}
			\alpha \circ F(\beta \circ \overline{f}) =  G(\psi) \circ \Delta,
		\end{equation}
		thus \((\beta \circ \overline{f}, \psi)\) is a morphism in \((F \backslash G)\). We will prove that \((\overline{f}, \overline{g}) : \operatorname{Coim}(f,g) \rightarrow \operatorname{Im}(f,g)\) is a morphism in \((F \backslash G)\). From \cref{eq_morf_ind_cat_B_co_comma}:
		\[
		G(\psi) \circ \Delta = G(\overline{g} \circ \Tilde{\theta}) \circ \Delta = G(\Tilde{\theta}) \circ G(\overline{g}) \circ \Delta.
		\]
		By \cref{eq_morf_im_co_comma}:
		\[
		\alpha \circ F(\beta \circ \overline{f}) = \alpha \circ F(\beta) \circ F(\overline{f}) = G(\Tilde{\theta}) \circ\Gamma \circ F(\overline{f}).
		\]
		From \cref{eq_morf_varphi_psi_co_comma}, we have
		\[
		G(\Tilde{\theta}) \circ G(\overline{g}) \circ \Delta =  G(\Tilde{\theta}) \circ \Gamma \circ F(\overline{f}).
		\]
		Since \(\Tilde{\theta}\) is an epimorphism and \(G\) is a left exact contravariant functor we conclude
		\begin{equation}
			\label{eq_mor_ind_co_comma_cand}
			G(\overline{g}) \circ \Delta =  \Gamma \circ F(\overline{f}),
		\end{equation}
		then \((\overline{f},\overline{g})\) is a morphism in \((F \backslash G)\). Moreover
		\[
		(f,g) = (\theta, \Tilde{\beta}) \circ (\overline{f}, \overline{g}) \circ (\beta, \Tilde{\theta}) = \operatorname{ker}(\operatorname{coker}(f,g)) \circ (\overline{f}, \overline{g}) \circ
		\operatorname{coker}(\operatorname{ker}(f,g)).
		\]
		Thus \((\overline{f}, \overline{g})\) is the induced morphism of \((f,g)\). Finally we will prove that \((\overline{f}, \overline{g})\) is an isomorphism. Since \(\mathcal{A}\) and \(\mathcal{B}\) are abelian we have \(\overline{f}\) and \(\overline{g}\) are isomorphisms. Thus, composing \(G((\overline{g})^{-1})\) on the left and \(F((\overline{f})^{-1})\) on the right in \cref{eq_mor_ind_co_comma_cand}, we obtain
		\[
		G((\overline{g})^{-1}) \circ \Gamma = \Delta \circ  F((\overline{f})^{-1}),
		\]
		thus \(((\overline{f})^{-1}, (\overline{g})^{-1})\) is a morphism in \((F \backslash G)\). Moreover,
		\[
		(\overline{f}, \overline{g}) \circ ((\overline{f})^{-1},(\overline{g})^{-1}) = ((\overline{f})^{-1} \circ \overline{f}, \overline{g} \circ (\overline{g})^{-1}) = (1_{\operatorname{Coim}f}, 1_{\operatorname{Im}g}) = 1_{\operatorname{Im}(f,g)} .
		\]
		\[
		((\overline{f})^{-1},(\overline{g})^{-1}) \circ (\overline{f}, \overline{g}) = (\overline{f} \circ (\overline{f})^{-1}, (\overline{g})^{-1} \circ \overline{g}) = (1_{\operatorname{Im}f}, 1_{\operatorname{Coim}g}) = 1_{\operatorname{Coim}(f,g)} .
		\]
		Hence, the induced morphism is an isomorphism. Therefore, \((F \backslash G)\) is an abelian category.	
	\end{proof}
	
	We can obtain a result analogous to \cref{teo_cosist_abeli} for the other variations of the definition of a comma category by making the necessary adaptations, as described below.
	
	\begin{theorem}
		Let $\mathcal{A}$, $\mathcal{B}$ and $\mathcal{C}$ be abelian categories, $F: \mathcal{A} \rightarrow \mathcal{C}$ a left exact contravariant functor and $G: \mathcal{B} \rightarrow \mathcal{C}$ a left exact functor. Then, \((F|G)\) is an abelian category.
	\end{theorem}
	
	\begin{theorem}
		Let $\mathcal{A}$, $\mathcal{B}$ and $\mathcal{C}$ be abelian categories, $F: \mathcal{A} \rightarrow \mathcal{C}$ a left exact contravariant functor and $G: \mathcal{B} \rightarrow \mathcal{C}$ a right exact contravariant functor. Then, \((F-G)\) is an abelian category.
	\end{theorem}
	
	\sloppy
	\printbibliography

@mastersthesis{miranda2023moduli,
	title={Moduli Spaces of Neural Networks},
	author={Miranda, Oscar Daniel Bernal},
	year={2023},
	school={Dissertação de Mestrado, IMECC-UNICAMP}
}

@book{mac2013categories,
	title={Categories for the Working Mathematician},
	author={Mac Lane, Saunders},
	volume={5},
	year={2013},
	publisher={Springer Science \& Business Media}
}

@book{freyd1964abelian,
	title={Abelian Categories},
	author={Freyd, Peter J},
	year={1964},
	publisher={Harper \& Row New York}
}

@article{rudakov1997stability,
	title = {Stability for an Abelian Category},
	author = {Alexei Rudakov},
	journal = {Journal of Algebra},
	volume = {197},
	number = {1},
	pages = {231--245},
	year = {1997}
}

@book{grandis2021category,
	title={Category Theory and Applications: A Textbook for Beginners},
	author={Grandis, Marco},
	year={2018},
	publisher={World Scientific}
}

@book{riehl2017category,
	title={Category Theory in Context},
	author={Riehl, Emily},
	year={2017},
	publisher={Courier Dover Publications}
}

@book{yekutieli2019derived,
	title={Derived Categories},
	author={Yekutieli, Amnon},
	volume={183},
	year={2019},
	publisher={Cambridge University Press}
}

@mastersthesis{uitinduced,
	title        = {Induced Maps on Grothendieck Groups},
	author       = {Niels uit de Bos},
	school       = {Universiteit Leiden},
	year         = {2014}
}

@book{rotman2009introduction,
	title={An introduction to homological algebra},
	author={Rotman, Joseph J},
	volume={2},
	year={2009},
	publisher={Springer}
}

@article{lu2013algebraic,
	title={Algebraic Structures on Grothendieck Groups},
	author={Lu, Weiyun and McBride, Aaron K},
	journal={Department of Mathematics and Statistics, University of Ottawa},
	year={2013}
}

@book{rotman2010advanced,
	title={Advanced Modern Algebra},
	author={Rotman, Joseph J},
	volume={114},
	year={2010},
	publisher={American Mathematical Soc.}
}

@article{macri2017lectures,
	title={Lectures on Bridgeland Stability},
	author={Macrì, Emanuele and Schmidt, Benjamin},
	journal={Moduli of Curves},
	series = {Lect. Notes Unione Mat. Ital.},
	volume = {21},
	pages={139--211},
	year={2017},
	publisher={Springer, Cham}
}

@book{le1993systemes,
	title={Systèmes cohérents et structures de niveau},
	author={Le Potier, Joseph},
	journal = {Astérisque},
	publisher = {Société Mathématique de France},
	volume = {214},
	year = {1993}
}

@book{weibel2013k,
	title={The $ K $-book: An Introduction to Algebraic $ K $-theory},
	author={Weibel, Charles A},
	volume={145},
	year={2013},
	publisher={American Mathematical Soc.}
}

@article{he1996espaces,
	title={Espaces de modules de systemes coh{\'e}rents},
	author={He, Min},
	journal={International Journal of Mathematics},
	volume={5},
	number={2},
	pages={545--598},
	year={1998},
	publisher = {World Scientific Publishing Company}
}

@misc{stacksproj,
	author       = {The {Stacks project authors}},
	title        = {The Stacks project},
	howpublished = {\url{https://stacks.math.columbia.edu/tag/0FCD}},
	year         = {2025},
}

@phdthesis{lawvere1963functorial,
	title={Functorial semantics of algebraic theories: And, some algebraic problems in the context of functorial semantics of algebraic theories},
	author={Lawvere, F William},
	year={1963},
	school = {Columbia University}
}

@article{lawvere2004reprints,
	title={Reprints in theory and applications of categories, no. 5},
	author={Lawvere, F William},
	journal={functorial semantics of algebraic theories and some algebraic problems in the context of functorial semantics of algebraic theories},
	pages={1--121},
	year={2004}
}

@misc{unapologetic2007comma,
	author       = {The Unapologetic Mathematician},
	title        = {Comma Categories},
	year         = {2007},
	note         = {Available at \url{https://unapologetic.wordpress.com/2007/05/26/comma-categories/}, accessed 22 August 2025}
}

@book{huybrechts2010geometry,
	title={The geometry of moduli spaces of sheaves},
	author={Huybrechts, Daniel and Lehn, Manfred},
	year={2010},
	publisher={Cambridge University Press}
}

@incollection{newstead2011,
	author    = {Newstead, P. E.},
	title     = {Existence of $\alpha$-stable coherent systems on algebraic curves},
	booktitle = {Grassmannians, Moduli Spaces and Vector Bundles},
	series    = {Clay Mathematics Proceedings},
	volume    = {14},
	pages     = {121--139},
	year      = {2011},
	publisher = {American Mathematical Society},
	address   = {Providence, RI}
}
\end{document}